\documentclass[10pt, A4paper]{article}

\usepackage{amsmath,amsthm,amssymb,amscd}
\usepackage[dvips]{graphicx}
\usepackage{color}

\allowdisplaybreaks

\newtheorem{theorem}{Theorem}[section]
\newtheorem{proposition}[theorem]{Proposition}
\newtheorem{lemma}[theorem]{Lemma}

\theoremstyle{definition}
\newtheorem{definition}[theorem]{Definition}
\newtheorem{remark}[theorem]{Remark}

\newcommand{\F}{\mathcal{F}}
\newcommand{\ff}{\mathbb{F}}
\newcommand{\Hz}{\mathcal{H}}

\newcommand{\R}{\mathbf{R}}
\newcommand{\N}{\mathbf{N}}

\begin{document}

\title
{Asymptotic error distributions of
the Euler method for continuous-time nonlinear filtering}

\author{Teppei Ogihara\footnote{The Institute of Statistical Mathematics, 
The Institute of Statistical Mathematics, 
10-3 Midori-cho, Tachikawa, Tokyo 190-8562, Japan, 
Tel: +81-50-5533-8500
mail: ogihara@ism.ac.jp} 
~and 
Hideyuki Tanaka\footnote{Corresponding author, 
National Institute of Technology, Toba College, 
1-1 Ikegami-cho, Toba City, Mie,517-8501, Japan, 
Tel: +81-599-25-8000, 
mail: h-tanaka@toba-cmt.ac.jp}
}

\date{\today}

\maketitle

\begin{abstract}
In this paper, 
we deduce the asymptotic error distribution 
of the Euler method for the nonlinear filtering problem with continuous-time observations. 
As studied in previous works by several authors, 
the error structure of the method is characterized by 
conditional expectations of some functionals of multiple stochastic integrals. 
Our main result is to prove 
the stable convergence of a sequence of such conditional expectations 
by using the techniques of martingale limit theorems in the spirit of Jacod (1997). 

\begin{flushleft}
\end{flushleft}

\begin{flushleft}
\end{flushleft}
\end{abstract}


\section{Introduction}\label{sec:introduction}

The functional central limit theorem for sequences of stochastic processes on path spaces 
has been developed with many applications. 
In the semimartingale setting, Jacod \cite{J97} established useful martingale limit theorems with mixed normal limits. 
They enabled us to study non-ergodic limits of statistics. 
In particular, they have been applied to statistical inference of diffusion processes observed at a fixed time interval. 
Among many studies, we refer the readers Gobet \cite{gob01} for diffusion processes with equi-distance observations, 
Jacod et al. \cite{jac-etal09} for pre-averaged estimators for volatility with noisy observations, 
and Ogihara and Yoshida \cite{ogi-yos14} for diffusion processes with nonsynchronous observations. 
When we consider these models of diffusions, statistics usually have non-ergodic limits, 
and martingale limit theorems in Jacod \cite{J97} are useful to deduce the asymptotic distributions. 
While there are many applications in statistical inference, 
the theorems are general enough to apply other problems related to semimartingales. 
We will apply the theorems to the Euler method for the nonlinear filtering problem. 

The purpose of this paper is to determine the asymptotic distribution of 
conditional expectations of some (Wiener) functionals and its application to 
the discrete-time approximation for nonlinear filtering with continuous-time observations. 
The key result is that, roughly speaking, for suitable (possibly multi-dimensional) 
$F$, $(\theta_t)_{t \in [0,1]}$, and a Brownian motion $(W_t)_{t \in [0,1]}$, 
\begin{align}\label{eq:target}
\sqrt{n} E\Big[F \int_0^1\int_{[sn]/n}^s \theta_r dW_r dW_s \Big| (W_t)_{t \in [0,1]}\Big]
\end{align}
converges stably in law to a mixed normal random variable; 
see Definition \ref{def:stable} and the precise formulation in Theorem \ref{thm:limittheorem}. 
This is a non-trivial task in the field of martingale limit theorem. 
From the reason why the interchange of the stable limit and conditional expectation operations is not allowed in some cases, 
a result in Jacod and Protter \cite{JP}, which gives the stable limit of $\sqrt{n}\int_0^1\int_{[sn]/n}^s \theta_r dW_r dW_s$, 
is not enough to show the limit theorem for (\ref{eq:target}). 
We will prove the desired limit theorem using some approximation techniques 
including the method of proof in \cite{JP}. 

Our goal in the rest of this paper is to deduce the asymptotic error distribution of the Euler method 
for nonlinear filtering (described in Theorem \ref{thm:app1} and Theorem \ref{thm:app2}). 
The objective of filtering theory is the estimation of an unobservable process $X=(X_t)_{t\in[0,1]}$ under the observation of 
$Y=(Y_t)_{t\in[0,1]}$. 
We assume $Y$ solves the following stochastic differential equation driven by a Brownian motion $(W_t)_{t\in[0,1]}$:
\[
Y_t = \int_0^t h(X_s, Y_s)ds + W_t.
\]
For the application in many fields, we want to know the conditional distribution of $X_1$ 
given the $\sigma$-field generated by $(Y_t)_{t \in [0,1]}$ (see e.g.\ \cite{BC09}, \cite{K11}, \cite{LS77}). 
We will assume later that 
$X$ is a stochastic process depending on $Y$ and another noise $B=(B_t)_{t\in[0,1]}$, 
in particular, $X$ solves a stochastic differential equation driven by 
$(B_t, Y_t)_{t\in[0,1]}$ and whose coefficients depend on $X_t$ and $Y_t$. 
To obtain an approximate value of the conditional expectation $\pi_1(g) := E[g(X_1)|(Y_t)_{t \in [0,1]}]$, 
one can use an approximation $\bar{\pi}_1^n(g)$ constructed by the Euler method (with the regular partition) for it. 
$L^p$-error bounds for $\pi_1(g) - \bar{\pi}_1^n(g)$ were obtained in several papers. 
In general, one can obtain its rate of convergence as follows (see Proposition \ref{prop:lpestimate2}): 
\begin{align}\label{eq:order_1/2}
\|\pi_1(g) - \bar{\pi}_1^n(g)\|_p \leq \frac{C}{\sqrt{n}}. 
\end{align}
In specific cases, the rate of convergence is known to be $1/n$ as $n \rightarrow \infty$; 
see e.g.~Picard \cite{P84}, Talay \cite{Ta86}, Milstein and Tretyakov \cite{MT09}: 
\begin{align}\label{eq:order1}
\|\pi_1(g) - \bar{\pi}_1^n(g)\|_p \leq \frac{C}{n}. 
\end{align}
We are interested in the asymptotic distribution of the error $\pi_1(g) - \bar{\pi}_1^n(g)$. 
Our objective is to apply the limit theorem for (\ref{eq:target}) 
to characterizing the asymptotic error distribution of the form 
\begin{align}\label{eq:finalerror_target}
\sqrt{n} (\pi_1(g) - \bar{\pi}_1^n(g)).
\end{align}
Of course, this value converges to zero if the estimate (\ref{eq:order1}) is satisfied.
However, in general, this converges to a non-zero random variable in the sense of stable convergence in law. 
In fact, if $\bar{\pi}_1^n(g)$ is given by the Euler-type approximation as in Section \ref{sec:applications} of the present paper, 
the sequence (\ref{eq:finalerror_target}) converges stably in law toward a mixed normal distribution 
whose conditional variance is obtained explicitly. 
The use of stable convergence is explained as follows. 
The error can be decomposed, in Theorem \ref{thm:app1} and Theorem \ref{thm:app2}, as 
\[
\sqrt{n} (\pi_1(g) - \bar{\pi}_1^n(g)) = a_n Z_n + a'_n Z'_n
\]
where the random variables $a_n$, $a'_n$, $Z_n$ and $Z'_n$ satisfy the following conditions (A) and (B):
\begin{description}
\item[(A)] Both $Z_n$ and $Z'_n$ can be represented by conditional expectations of some functionals (such as (\ref{eq:target})). 
Moreover $(Z_n,Z'_n)$ converges stably in law to some random variable $(Z_\infty, Z'_\infty)$. 
\item[(B)] $a_n \rightarrow a_\infty$ and $a'_n \rightarrow a'_\infty$ in probability 
for some random variables $a_\infty$ and $a_\infty'$.
\end{description}
Then $a_n Z_n + a'_n Z'_n$ still converges stably in law, to $a_\infty Z_\infty+a'_\infty Z'_\infty$. 
Note that if convergence in (A) is weaker (e.g., convergence in law), 
then the limit distribution of $a_n Z_n + a'_n Z'_n$ is no longer obtained. 
To complete the proofs, the most important consideration is 
to understand how the limit theorem for (\ref{eq:target}) can be applied to the results of type (A). 
For the purpose, we apply several techniques seen in Jacod and Protter \cite{JP} and Clement et al.~\cite{CKL06}. 

The article is organized as follows. 
In Section \ref{sec:pre}, we give a brief review of stable convergence 
and its fundamental properties. 
Section \ref{sec:limittheorem} contains the statements and proofs of our main results 
in terms of stable convergence of conditional law. 
The main tool of the derivation is maritingale limit theorems in \cite{J97}. 
In Section \ref{sec:applications}, an application of the limit theorems to 
the Euler method for continuous-time nonlinear filtering is discussed. 
We rely on the ``change of measure'' approach via Girsanov's theorem 
in order to construct numerical schemes. 
First, we deal with a simple case where 
the Euler method is only applied to the Girsanov's density. 
Second, we tackle a more complicated problem 
in the Euler method for both $(X_t)_{t \in [0,1]}$ and the Girsanov's density. 

\section{Preliminaries}\label{sec:pre}
\subsection{Stable convergence in law}
Let $(\Omega,\F,P)$ be a complete probability space and $E$ be a metric space. 
The space $b\mathcal{G}$ denotes the set of all bounded $\mathcal{G}$-measurable real-valued functions for a given 
sub $\sigma$-field $\mathcal{G} \subset \F$. 
Let us consider $E$-valued random variables $(X_n)_{n\in\N}$ and $X$. 
We denote several types of convergence
by the following notation. 
\begin{description}
\item[(i)] $X_n \rightarrow^P X$ : Convergence in probability (under the probability measure $P$).
\item[(ii)] $X_n = o_P(1)$ is defined as $X_n \rightarrow^P 0$.
\item[(iii)] $X_n \Longrightarrow^{\mathcal{L}} X$ : Convergence in law (weak convergence).
\item[(iv)] $X_n \Longrightarrow^{s-\mathcal{G}} X$ : $\mathcal{G}$-stable convergence in law with $\mathcal{G} \subset \F$.
\end{description}

We now define precisely $\mathcal{G}$-stable convergence in law below. 

\begin{definition}\label{def:stable}
Let $\mathcal{G}$ be a sub $\sigma$-field of $\F$. 
For $E$-valued random variables $(X_n)_{n\in\N}$,  
we say $(X_n)_{n\in\N}$ converges $\mathcal{G}$-stably to $X$ 
(i.e.\ $X_n \Longrightarrow^{s-\mathcal{G}} X$) 
if $X$ is an $E$-valued random variable on 
an extension $(\hat{\Omega},\hat{\F}, \hat{P})$ of $(\Omega, \F, P)$ 
and 
\begin{equation}\label{eq:stable}
\lim_{n\rightarrow\infty}E[f(X_n)Y] = \hat{E}[f(X)Y]
\quad \textrm{ for all } f \in C_b(E), Y \in b\mathcal{G}.
\end{equation}
\end{definition}

The extension probability space $(\hat{\Omega},\hat{\F}, \hat{P})$ is assumed that 
$(\hat{\Omega},\hat{\F}) = (\Omega, \F)\times (\Omega_0, \F_0)$ and 
$\hat{P}(A \times \Omega_0) = P(A), A \in \F$ for some measurable space $(\Omega_0, \F_0)$. 
See Aldous and Eagleson \cite{ald-eag78} or Jacod and Shiryaev \cite{JS2} for fundamental properties of stable convergence. 
For the readers, we state some well-known facts.  
\begin{itemize}
\item If $X_n \Longrightarrow^{s-\F} X$ and $Y_n \rightarrow^P Y$, then $(X_n,Y_n) \Longrightarrow^{s-\F} (X,Y)$.
\item Let $X$ be defined on the original space $(\Omega, \F, P)$. Then 
$X_n \Longrightarrow^{s-\F} X$ if and only if $X_n \rightarrow^P X$. 
\item Recall that the limit (\ref{eq:stable}) still holds for any $Y \in L^1(\mathcal{G})$. 
This yields that the $\F$-stable convergence remains for any absolutely continuous measure $Q$ with respect to $P$. 
In fact, we can take an extension $(\hat{\Omega},\hat{\F}, \hat{Q})$ with $d\hat{Q} = \frac{dQ}{dP}d\hat{P}$. 
\item A version of the Portmanteau Theorem (see e.g.~\cite{B99}) below holds. 
\end{itemize}

\begin{lemma}\label{lem:Portmanteau}
Fix an extended probability space $(\hat{\Omega},\hat{\F}, \hat{P})$ and a sub $\sigma$-field $\mathcal{G} \subset \F$. 
Then the following conditions are equivalent. 
\begin{itemize}
\item[{\rm(1)}] $\displaystyle \lim_{n\rightarrow\infty}E[f(X_n)Y] = \hat{E}[f(X)Y] 
\textrm{ for all } f \in C_b(E)$ and $Y \in b\mathcal{G}$.
\item[\rm{(2)}] $\displaystyle \lim_{n\rightarrow\infty}E[f(X_n)Y] = \hat{E}[f(X)Y] 
\textrm{ for all bounded uniformly continuous function } f$ and $Y \in b\mathcal{G}$.
\item[\rm{(3)}] $\displaystyle \limsup_{n\rightarrow\infty}E[1_{X_n \in F}Y] \leq \hat{E}[1_{X\in F}Y] 
\textrm{ for all closed set } F$ and $Y \in b\mathcal{G}$ with $Y \geq 0$.
\item[\rm{(4)}] $\displaystyle \hat{E}[1_{X\in G}Y]  \leq \liminf_{n\rightarrow\infty}E[1_{X_n \in G}Y]
\textrm{ for all open set } G$ and $Y \in b\mathcal{G}$ with $Y \geq 0$.
\end{itemize}
\end{lemma}

\begin{proof}
It is possible to show (1) $\Rightarrow$ (2) $\Rightarrow$ (3) $\Leftrightarrow$ (4) $\Rightarrow$ (1) 
as shown in the case of convergence in law. 
\end{proof}

We prove here an auxiliary lemma applied to limit theorems in the next section. 
\begin{lemma}\label{lem:sublimit}
Let $\mathcal{G}$ be a sub $\sigma$-field of $\F$ and 
$(X_n)_{n \in \N}$ be $E$-valued $\F$-measurable random variables. 
Assume that, there exist an extension $(\hat{\Omega}, \hat{\F}, \hat{P})$ of original probability space, 
$E$-valued $\F$-measurable random variables $(X_n^m)_{n,m \in \N}$, and 
$E$-valued $\hat{\F}$-measurable random variables 
$(X^m)_{m \in \N}, X$
such that the following conditions hold. 
\begin{itemize}
\item[{\rm (a)}] For any $\varepsilon >0$,  $\displaystyle \lim_{m\rightarrow \infty} \limsup_{n\rightarrow\infty} P( |X_n - X_n^m| > \varepsilon ) = 0$.
\item[{\rm (b)}] $\displaystyle \lim_{n \rightarrow \infty}E[f(X_n^m)Y] = \hat{E}[f(X^m)Y] $ for each $m \in \N$, $f \in C_b(E)$ and $Y \in b\mathcal{G}$. 
\item[{\rm (c)}] For any $\varepsilon >0$,  $\displaystyle \lim_{m\rightarrow \infty} \hat{P}( |X^m - X| > \varepsilon ) = 0$.
\end{itemize}
Then $X_n \Longrightarrow^{s-\mathcal{G}} X$. 
\end{lemma}
\begin{proof}
We shall prove the condition (2) in Lemma \ref{lem:Portmanteau}. 
Let $f$ be a bounded uniformly continuous function and $Y \in b\mathcal{G}$. First we have 
\begin{align*}
& |E[f(X_n)Y]-\hat{E}[f(X)Y]| 
\\ & \leq |E[f(X_n^m)Y]-E[f(X_n)Y]| + |E[f(X_n^m)Y]-\hat{E}[f(X^m)Y]|  + |\hat{E}[f(X^m)Y]-\hat{E}[f(X)Y]|.
\end{align*}
By the assumption (b), 
\begin{align*}
\limsup_{n\rightarrow \infty}|E[f(X_n)Y]-\hat{E}[f(X)Y]| \leq \limsup_{n\rightarrow\infty} |E[f(X_n^m)Y]-E[f(X_n)Y]| +  |\hat{E}[f(X^m)Y]-\hat{E}[f(X)Y]|. 
\end{align*}
Since $f$ is uniformly continuous, for a given $\varepsilon >0$ we can choose $\delta >0$ 
such that $|x-y| < \delta$ implies $|f(x)-f(y)|<\varepsilon$ for every $x,y \in E$. Hence we have 
\begin{align*}
|E[f(X_n^m)Y]-E[f(X_n)Y]|
& \leq E[|(f(X_n^m)-f(X_n))Y| (1_{|X_n^m-X_n|>\delta}+1_{|X_n^m-X_n|\leq \delta})]
\\ & \leq 2 \|f\|_\infty \|Y\|_\infty P( |X_n - X_n^m| > \delta) + \varepsilon \|Y\|_\infty.
\end{align*}
Since $\varepsilon$ and $m$ are arbitrary, this yields  
$\limsup_{n\rightarrow \infty}|E[f(X_n)Y]-\hat{E}[f(X)Y]| = 0$
by (a) and (c).
\end{proof}

\subsection{Some properties of stochastic integrals under conditional probability}
For a stochastic process $X=(X_t)_{t \in [0,1]}$ on $(\Omega, \F, P)$, 
we define the filtration generated by $X$ as 
\[
\F_t^X := \sigma(X_s: 0\leq s \leq t) \vee \{\textrm{all } P\textrm{-null sets} \}.
\]
Let $W=(W_t)_{t \in [0,1]}$ be a $d$-dimensional standard Brownian motion. 
We consider some filtrations $\ff = (\F_t)_{t\in[0,1]}$ so that $W$ becomes 
a $\ff$-Brownian motion for the applications to nonlinear filtering. 
For example, in Section \ref{sec:applications} we specify 
\begin{align*}
\F_t = \sigma(X_0) \vee \F_t^B \vee \F_t^W, 
\quad  \textrm{ or } \quad \F_t = \sigma(X_0) \vee \F_1^B \vee \F_t^W, 
\end{align*}
where 
$B$ is a standard Brownian motion independent of $W$ and 
$X_0$ is a initial random variable of a stochastic differential equation. 

We introduce a Fubini-type theorem for stochastic integrals and conditional expectations. 
We fix a filtration $\ff$ satisfying the usual conditions. 
\begin{lemma}[See e.g.\ \cite{BC09}]\label{lem:fubini-type}
Let $(W_t)_{t\in[0,1]}$ be a $d$-dimensional $\ff$-Brownian motion and 
$(f_t)_{t\in[0,1]}$ be a $\R^d$-valued predictable process with respect to 
$\ff$ with
$E[\int_0^1 |f_s|^2 ds] < \infty$.
Then
\begin{description}
\item[{\rm (i)}]
 $(E[f_t| \F_t^W])_{t\in[0,1]}$ has a predictable version and satisfies that for $s \leq 1$, 
\[
E[f_s|\F_s^W](\omega) = E[f_s|\F_1^W](\omega) \quad \mbox{ a.e.-} (s,\omega).
\]

\item[{\rm (ii)}] 
\[
E\Big[\int_0^1 f_s dW_s\Big|\F_1^W\Big] = \int_0^1 E[f_s| \F_1^W] dW_s = \int_0^1 E[f_s| \F_s^W] dW_s \quad \mbox{ a.s.}
\]
\end{description}
\end{lemma}

\begin{proof}
The conditions (i) and (ii) are clearly satisfied in the case where $f$ is a simple process such as 
\[
f_s = \sum_{i=1}^m A_i 1_{(t_i, t_{i+1}]}(s)
\]
with $t_1 < \cdots < t_{m}$ and $F_{t_i}$-measurable random variables $(A_i)$. 
The result follows from taking its limit. 
\end{proof}

It is worth noticing that 
the stable limit and conditional expectation operations cannot be interchanged in general. 
For instance, let $W$ and $B$ be one-dimensional Brownian motions independent of each other 
and $(\theta_s)$ be a square-integrable predictable process. 
A martingale limit theorem in \cite{JP} shows that 
$\sqrt{n}(\int_0^1\int_{[sn]/n}^s \theta_r dW_r dW_s, \int_0^1\int_{[sn]/n}^s \theta_r dB_r dW_s)$ 
converges $\F$-stably to $\frac{1}{\sqrt{2}}( \int_0^1\theta_s N_s^1,  \int_0^1\theta_s N_s^2)$ 
with a two-dimensional Brownian motion $(N_t) = (N_t^1, N_t^2)$ independent of $\F$. 
Clearly each component of the limit has the same law. 
However one can easily show that 
$\sqrt{n} E[\int_0^1\int_{[sn]/n}^s \theta_r dW_r dW_s | \F_1^W]$ 
and 
$\sqrt{n} E[\int_0^1\int_{[sn]/n}^s \theta_r dB_r dW_s | \F_1^W]$ 
no longer have the same limit. Indeed the pair of the above conditional expectations converges $\F$-stably to 
$(\frac{1}{\sqrt{2}}\int_0^1 E[\theta_s| \F_s^W]dN_s^1, 0)$ by Lemma \ref{lem:fubini-type}.

\section{Limit theorems}\label{sec:limittheorem}
In this section we develop limit theorems for 
certain sequences of conditional expectations. 
Let $B=(B_t)_{t\in[0,1]}$ and $W=(W_t)_{t\in[0,1]}$ be $e$- 
and $d$-dimensional standard Brownian motions defined on 
a complete probability space $(\Omega,\F,P)$. 
We use a step function $\eta_n$ defined by
$\eta_n(t)=[tn]/n, t \in [0,1]$. 

The main result in this section is as follows. 
\begin{theorem}\label{thm:limittheorem}
Let $\Hz \subset \F$ be a sub $\sigma$-field independent of $W$ and set 
$\ff = (\F_t)_{t\in[0,1]}$ where $\F_t := \Hz \vee \F_t^W$, $t \in [0,1]$. 
If $F = (F^1, \dots, F^q) \in L^2(\F_1; \R^q)$ and 
$(\theta_t^{ijk})_{t\in[0,1]}$, $1 \leq i,j \leq d$, $1 \leq k \leq q$ are $\ff$-predictable processes 
satisfying that $E[\int_0^1 |\theta_s^{ijk}|^2ds] < \infty$ for all $i,j,k$, 
then 
\begin{align*}
& \left( \sqrt{n} E\Big[F^k \int_0^1 \int_{\eta_n(s)}^s \theta_r^{ijk} dW_r^j dW_s^i \Big|\F_1^W\Big] \right)_{1 \leq i,j \leq d, 1 \leq k \leq q}
\\ & \Longrightarrow^{s-\F}\left( \frac{1}{\sqrt{2}} \int_0^1 E[F^k\theta_s^{ijk}|\F_1^W] dN_s^{ij} \right)_{1 \leq i,j \leq d, 1 \leq k \leq q}
\end{align*}
where $\{(N_t^{ij})_{t\in[0,1]}: 1\leq i,j\leq d\}$ is a $d^2$-dimensional standard Brownian motion independent of $\F$. 
The stochastic integral in the limit is well-defined 
with respect to the filtration $(\F_1 \vee \F_t^N)_{t\in[0,1]}$. 
\end{theorem}

\begin{proof}
We prove here only the case $q=1$ for notational simplicity. Let us denote $F := F^1$ and $\theta^{ij} := \theta^{ij1}$. 
We first define $X_n =(X_n(i,j))_{1 \leq i,j \leq d}$ by 
\[
X_n(i,j) := \sqrt{n} E\Big[F \int_0^1 \int_{\eta_n(s)}^s \theta_r^{ij} dW_r^j dW_s^i \Big|\F_1^W\Big]. 
\]
In what follows, we consider an extended probability space 
\[
(\hat{\Omega},\hat{\F},\hat{P}) := (\Omega,\F,P) \times \prod_{1\leq i,j \leq d} (\Omega^{ij},\F^{ij},P^{ij})
\]
where, for each $1 \leq i,j\leq d$, $(\Omega^{ij},\F^{ij},P^{ij})$ is the Wiener space on which $N^{ij}=(N_t^{ij})_{t \in [0,1]}$ 
is the canonical Brownian motion. 
Set $X =(X(i,j))_{1 \leq i,j \leq d}$ as 
\[
X(i,j) = \frac{1}{\sqrt{2}} \int_0^1 E[F \theta_s^{ij} |\F_1^W] dN_s^{ij}.
\]
Here the above stochastic integral is defined with respect to the filtration $(\F_1^W\vee\F_t^N)_{t\in[0,1]}$. 
The idea of the proof of the theorem can be sketched as 
the following diagram stated in Lemma \ref{lem:sublimit} (with $E = \R^{d^2}$)
\[
\begin{CD}
X_n @>>> X\\
@A{(a)}AA @AA{(c)}A \\
X_n^m @>>{(b) \text{ stable limit}}> X^m
\end{CD}
\]
where $(X_n^m)$ and $(X^m)$ are approximation sequences defined below 
and (a), (b) and (c) indicate the convergence in Lemma \ref{lem:sublimit}. 

Step $1$: The goal of this step is to find a sequence $X_n^m$ 
which satisfies the convergence (a) 
and is also suitable for taking the limit (b) and (c). 
Since $F$ is $\Hz \vee \F_1^W$-measurable,  
we can choose an approximation sequence $F_m$ given by 
\begin{align*}
F_m = \sum_{l=1}^m \Psi_{ml}^{\Hz}\Psi_{ml}^W, \quad \Psi_{ml}^{\Hz} \in b\Hz, \Psi_{ml}^W \in b\F_1^W
\end{align*}
and satisfying 
\begin{align*}
\lim_{m\rightarrow\infty} \|F - F_m\|_2 = 0. 
\end{align*}
The auxiliary process $X_n^m =(X_n^m(i,j))_{1 \leq i,j \leq d}$ is defined by 
\[
X_n^m(i,j) := \sqrt{n} E\Big[F_m \int_0^1 \int_{\eta_n(s)}^s \theta_r^{ij} dW_r^j dW_s^i \Big|\F_1^W\Big].
\]
By the Cauchy-Schwarz inequality, we have 
\begin{align*}
\sup_n \|X_n^m(i,j) - X(i,j)\|_1 
& \leq \sup_n \Big\|E\Big[(F-F_m) \sqrt{n} \int_0^1 \int_{\eta_n(s)}^s \theta_r^{ij} dW_r^j dW_s^i\Big|\F_1^W\Big] \Big\|_1 
\\ &\leq \|F - F_m\|_2 \Big\| \sqrt{n} \int_0^1 \int_{\eta_n(s)}^s \theta_r^{ij} dW_r^j dW_s^i \Big\|_2. 
\end{align*}
Using It\^o's isometry, we obtain 
\begin{align*}
\Big\| \sqrt{n} \int_0^1 \int_{\eta_n(s)}^s \theta_r^{ij} dW_r^j dW_s^i \Big\|_2^2 
& = E\Big[ n \int_0^1 \int_{\eta_n(s)}^s |\theta_r^{ij}|^2 dr ds\Big] 
\\ & \leq \int_0^1 E[|\theta_s^{ij}|^2] ds. 
\end{align*}
This yields $\sup_n \|X_n^m - X_n\|_1 \rightarrow 0$ as $m \rightarrow \infty$, 
hence $X_n^m$ satisfies the desired condition (a). 

Step $2$: Let $m$ be fixed. We next consider the stable limit (b) of $X_n^m$ for each $m$. 
We note that the discussion in this step will be based on the martingale limit theorems provided in \cite{J97} and \cite{JP}. 
We can show by the independence of $\Hz$ and $\F_1^W$ that 
\begin{align*}
\sqrt{n} E\Big[F_m \int_0^1 \int_{\eta_n(s)}^s \theta_r^{ij} dW_r^j dW_s^i\Big|\F_1^W\Big] 
&= \sqrt{n} \sum_{l=1}^m \Psi_{ml}^W E\Big[\Psi_{ml}^{\Hz} \int_0^1 \int_{\eta_n(s)}^s \theta_r^{ij} dW_r^j dW_s^i\Big|\F_1^W\Big]
\\ 
&= \sum_{l=1}^m \Psi_{ml}^W \sqrt{n} \int_0^1 \int_{\eta_n(s)}^s E[\Psi_{ml}^{\Hz} \theta_r^{ij} |\F_r^W] dW_r^j dW_s^i. 
\end{align*}
Set $\xi_t^{lij} := E[\Psi_{ml}^{\Hz} \theta_t^{ij} |\F_t^W]$. 
As seen in \cite[pp 290-293]{JP}, we obtain for each $t >0$, 
\begin{align*}
\sqrt{n} \Big\langle \int_0^{\cdot} \int_{\eta_n(s)}^s \xi_r^{lij}dW_r^jdW_s^i, W_{\cdot}^{i'} \Big\rangle_t 
&= \delta_{ii'} \sqrt{n} \int_0^t \int_{\eta_n(s)}^s \xi_r^{lij}dW_r^j ds \rightarrow^P 0, 
\\ n \Big\langle \int_0^{\cdot} \int_{\eta_n(s)}^s \xi_r^{lij}dW_r^jdW_s^i, \int_0^{\cdot} \int_{\eta_n(s)}^s \xi_r^{l'i'j'}dW_r^{j'}dW_s^{i'} \Big\rangle_t 
 & = \delta_{ii'} n \int_0^t \Big(\int_{\eta_n(s)}^s \xi_r^{lij}dW_r^j \int_{\eta_n(s)}^s \xi_r^{l'i'j'}dW_r^{j'}\Big) ds 
\\ & \rightarrow^P \frac{1}{2} \delta_{ii'}\delta_{jj'}\int_0^t \xi_s^{lij} \xi_s^{l'i'j'} ds. 
\end{align*}
These imply the stable convergence of the following $(m\times d^2)$-dimensional process 
in Skorohod space (see \cite[Theorem 2.1]{J97}). 
\begin{align}
& \Big(\sqrt{n} \int_0^\cdot \int_{\eta_n(s)}^s E[\Psi_{ml}^{\Hz} \theta_r^{ij} |\F_r^W] dW_r^j dW_s^i \Big)_{1\leq l \leq m, 1\leq i,j \leq d}
\nonumber 
\\ & \Longrightarrow^{s-\F} 
\Big(\frac{1}{\sqrt{2}} \int_0^\cdot E[\Psi_{ml}^{\Hz} \theta_s^{ij} |\F_s^W] dN_s^{ij}\Big)_{1\leq l \leq m, 1\leq i,j \leq d.}
\label{eq:jpresults}
\end{align}
By the definition of stable convergence, in particular, 
the result (\ref{eq:jpresults}) implies that 
\begin{align*}
& \Big(\sum_{l=1}^m \Psi_{ml}^W \sqrt{n} \int_0^1 \int_{\eta_n(s)}^s E[\Psi_{ml}^{\Hz} \theta_r^{ij} |\F_r^W] dW_r^j dW_s^i 
\Big)_{1\leq i,j \leq d}
\\ & \Longrightarrow^{s-\F} 
\Big(\sum_{l=1}^m \Psi_{ml}^W \frac{1}{\sqrt{2}} \int_0^1 E[\Psi_{ml}^{\Hz} \theta_s^{ij} |\F_s^W] dN_s^{ij} 
\Big)_{1\leq i,j \leq d.}
\end{align*}
By Lemma \ref{lem:fubini-type}, we get $E[\Psi_{ml}^{\Hz} \theta_s^{ij} |\F_s^W] = E[\Psi_{ml}^{\Hz} \theta_s^{ij} |\F_1^W]$ a.e.-$(s.\omega)$. 
Hence we have 
\begin{align*}
\sum_{l=1}^m \Psi_{ml}^W \frac{1}{\sqrt{2}} \int_0^1 E[\Psi_{ml}^{\Hz} \theta_s^{ij} |\F_s^W] dN_s^{ij}
& = \frac{1}{\sqrt{2}} \int_0^1 E\Big[\sum_{l=1}^m \Psi_{ml}^{\Hz} \Psi_{ml}^W \theta_s^{ij} \Big|\F_1^W\Big] dN_s^{ij}
\\ & = 
\frac{1}{\sqrt{2}} \int_0^1 E[F_m \theta_s^{ij} |\F_1^W] dN_s^{ij}
\\ & =: X^m(i,j). 
\end{align*} 
So we get (b). That is, $X_n^m \Longrightarrow^{s-\F} X^m = (X^m(i,j))_{1\leq i,j \leq d}$ for each $m$ 
in the extended probability space $(\hat{\Omega},\hat{\F},\hat{P})$. 

Step $3$: In order to prove (c), we will show that 
$X^m$ converges to $X$ in $L^1(\hat{\Omega},\hat{\F},\hat{P})$. 
By the Burkholder-Davis-Gundy inequality, $(X^m-X)$ satisfies 
\[
\|X^m(i,j)-X(i,j)\|_{L^1(\hat{\Omega}, \hat{\F}, \hat{P})} 
\leq C E\Big[\Big( \int_0^1 |E[(F_m-F) \theta_s^{ij} |\F_1^W]|^2 ds \Big)^{1/2}\Big].
\]
Using the Cauchy-Schwarz inequality for conditional expectations, we have
\[
\int_0^1 | E[(F-F_m) \theta_s^{ij} |\F_1^W] |^2 ds \leq E[(F-F_m)^2|\F_1^W] \int_0^1 E[|\theta_s^{ij}|^2|\F_1^W] ds
\]
Since $\|F_m-F\|_2 \rightarrow 0$, we have $\|E[(F-F_m)^2|\F_1^W]\|_1 \rightarrow 0$. Thus, 
\begin{align*}
X^m(i,j) = \frac{1}{\sqrt{2}} \int_0^1 E[F_m \theta_s^{ij} |\F_1^W] dN_s^{ij}
\rightarrow \frac{1}{\sqrt{2}} \int_0^1 E[F \theta_s^{ij} |\F_1^W] dN_s^{ij}
\end{align*}
in $L^1(\hat{\Omega}, \hat{\F}, \hat{P})$. Hence we get the claim (c). 

From Steps $1$ to $3$, we obtain $X_n \Longrightarrow^{s-\F} X$ 
in the extended probability space $(\hat{\Omega},\hat{\F},\hat{P})$ by virtue of Lemma \ref{lem:sublimit}. 
This finishes the proof. 
\end{proof}

\begin{remark}\label{rem:important}
For the later use in Section \ref{subsec:proof2}, 
we should mention that Theorem \ref{thm:limittheorem} can be modified to 
\begin{align*}
&\left( \sqrt{n} E\Big[F^k \int_0^1 \Big(\int_{\eta_n(s)}^s \theta_r^{ijk} dW_r^j\Big)
 \lambda_s^{ijk} dW_s^i \Big|\F_1^W\Big] \right)_{1 \leq i,j \leq d, 1 \leq k \leq q}
\\ & \Longrightarrow^{s-\F}\left( \frac{1}{\sqrt{2}} \int_0^1 E[F\theta_s^{ijk}\lambda_s^{ijk}|\F_1^W] dN_s^{ij} \right)_{1 \leq i,j \leq d, 1 \leq k \leq q}
\end{align*}
where $\theta^{ijk}$ and $\lambda^{ijk}$ are predictable processes which have finite moments of all order. 
Clearly, if $t \mapsto \lambda_t^{ijk}$ is bounded left continuous, then we get 
\begin{align}\label{eq:diff_integ}
\sqrt{n} \int_0^1 \Big(\int_{\eta_n(s)}^s \theta_r^{ijk} dW_r^j\Big)\lambda_s^{ijk} dW_s^i 
= \sqrt{n} \int_0^1 \Big(\int_{\eta_n(s)}^s \theta_r^{ijk}\lambda_r^{ijk} dW_r^j\Big) dW_s^i + o_P(1), 
\end{align}
where the remainder corresponding to $o_P(1)$ is uniformly integrable. 
In general cases, the result is proved by an approximation of $\lambda_t^{ijk}$ by simple processes. 
\end{remark}

We next derive limit theorems for the same type of conditional expectations 
\[
E[F \int_0^1\int_{\eta_n(s)}^s \theta_s dX_s dY_s|\F_1^W]
\]
so that at least one of $(X,Y)$ is not $W^i$. 
We use the notation $W_t^0 = B_t^0 = t$ for the simplicity of statements below. 

\begin{theorem}\label{thm:zerolimit2}
Let $\Hz$ be a sub $\sigma$-field of $\F$, and $\Hz, \F_1^B, \F_1^W$ are independent. 
Set $\ff = (\F_t)_{t\in[0,1]}$ where $\F_t := \Hz \vee \F_t^B \vee \F_t^W$, $t \in [0,1]$. 
Suppose that $F \in L^2(\F_1)$ and $(\theta_t)_{t\in[0,1]}$ is an $\ff$-predictable process with $E[\int_0^1\theta_t^2dt]<\infty$.  
Then we have for any $0 \leq i,j \leq d$, 
\begin{align}
& \sqrt{n} E\Big[F \int_0^1 \int_{\eta_n(s)}^s \theta_r dB_r^j dB_s^i\Big|\F_1^W\Big] \rightarrow 0, \label{eq:p1}
\\ & \sqrt{n} E\Big[F \int_0^1 \int_{\eta_n(s)}^s \theta_r dB_r^j dW_s^i\Big|\F_1^W\Big] \rightarrow 0, \label{eq:p2}
\\ & \sqrt{n} E\Big[F \int_0^1 \int_{\eta_n(s)}^s \theta_r dW_r^j dB_s^i\Big|\F_1^W\Big] \rightarrow 0, \label{eq:p3}
\end{align}
in probability and in $L^1$ as $n \rightarrow \infty$.
\end{theorem}

\begin{proof}
Let $\tilde{\theta}_t = \tilde{\theta}_t(K)$ be a bounded process defined as 
\[
  \tilde{\theta}_t = 
  \left\{ \begin{array}{ll}
    \theta_t, & -K \leq \theta_t \leq K \\
    K,  & \theta_t > K \\
    -K, & \theta_t < -K
  \end{array} \right.
\]
First we consider an $L^2$-approximation sequence $(F_m)$ for $F \in L^2(\F_1)$. 
For $X,Y = B^i$ or $W^j$ (at least one of them is $B^i$, $0\leq i \leq d$), we have 
\begin{align*}
\Big\|\sqrt{n} E\Big[F \int_0^1 \int_{\eta_n(s)}^s \theta_r dX_r dY_s\Big|\F_1^W\Big]\Big\|_1 
&\leq \|F - F_m\|_2 \Big(\int_0^1 E[|\theta_s|^2] ds\Big)^{\frac{1}{2}} 
\\ & \quad + \|F_m\|_2 \Big(\int_0^1 E[|\theta_s -\tilde{\theta}_s|^2] ds\Big)^{\frac{1}{2}} 
\\ & \quad + \Big\|\sqrt{n} E\Big[F_m \int_0^1 \int_{\eta_n(s)}^s \tilde{\theta}_r dX_r dY_s\Big|\F_1^W\Big]\Big\|_1. 
\end{align*}
From the above inequality, it suffices to construct $(F_m)_{m\in\N}$ such that $\|F - F_m\|_2 \rightarrow 0$ 
as $m\rightarrow \infty$ and 
\begin{align*}
\lim_{n\rightarrow \infty}\Big\|\sqrt{n} E\Big[F_m \int_0^1 \int_{\eta_n(s)}^s \tilde{\theta}_r dX_r dY_s\Big|\F_1^W\Big]\Big\|_1 = 0, 
\quad \forall m \in \N, \ \forall K>0.
\end{align*} 
We now construct $F_m$ as follows. Since $F$ is $\Hz \vee \F_1^B \vee \F_1^W$-measurable, 
we can get an $L^2$-approximation sequence $(F_m)_{m\in\N}$ of the form 
\[
F_m = \sum_{l=1}^m \Psi_{ml}^{\Hz}\Psi_{ml}^B\Psi_{ml}^W, 
\quad \Psi_{ml}^{\Hz} \in b\Hz, \Psi_{ml}^B \in b\F_1^B, \Psi_{ml}^W \in b\F_1^W.
\]
Substituting this into $F_m$, we obtain 
\[
E\Big[F_m \int_0^1 \int_{\eta_n(s)}^s \tilde{\theta}_r dX_r dY_s\Big|\F_1^W\Big] = 
\sum_{l=1}^m \Psi_{ml}^W E\Big[ \Psi_{ml}^{\Hz}\Psi_{ml}^B \int_0^1 \int_{\eta_n(s)}^s \tilde{\theta}_r dX_r dY_s\Big|\F_1^W\Big]. 
\]
Hence our goal is to prove 
\begin{align*}
\lim_{n\rightarrow \infty}\Big\|\sqrt{n} E\Big[\Psi^{\Hz}\Psi^B \int_0^1 \int_{\eta_n(s)}^s \tilde{\theta}_r dX_r dY_s\Big|\F_1^W\Big]\Big\|_1 = 0, 
\quad \Psi^{\Hz} \in b\Hz, \Psi^B \in b\F_1^B, 
\end{align*}
which is equivalent to the condition 
\begin{align}\label{eq:prop_goal}
\sqrt{n} E\Big[\Psi^{\Hz}\Psi^B \int_0^1 \int_{\eta_n(s)}^s \tilde{\theta}_r dX_r dY_s\Big|\F_1^W\Big] \rightarrow^P 0, 
\quad \Psi^{\Hz} \in b\Hz, \Psi^B \in b\F_1^B. 
\end{align}
Here we used the fact that the sequence $(\sqrt{n}\int_0^1\int_{\eta_n(s)}^s \tilde{\theta}_r dX_r dY_s)_{n}$ is $L^2$-bounded. 

In the case $X_t = t$, the limit (\ref{eq:prop_goal}) follows from the Cauchy-Schwarz inequality. 
In other cases, we can easily prove $\sqrt{n}E[\Psi^{\Hz}\int_0^1 \int_{\eta_n(s)}^s \tilde{\theta}_r dX_r dY_s|\F_1^W] \rightarrow^P 0$. 
Therefore without loss of generality, we may assume $E[\Psi^B]=0$ and $\Psi^B = \sum_{l=1}^d \int_0^1 f_s^l dB_s^l$ (by It\^o's representation theorem). 

Case $X_t = W_t^j$ $(1\leq j \leq d)$, $Y_t = t$: 
\begin{align*}
\sqrt{n} E\Big[\Psi^{\Hz}\Psi^B \int_0^1 \int_{\eta_n(s)}^s \tilde{\theta}_r dW_r^j ds\Big|\F_1^W\Big] 
& = \sqrt{n} \int_0^1 \int_{\eta_n(s)}^s E[\Psi^{\Hz}\Psi^B \tilde{\theta}_r |\F_r^W] dW_r^j ds
\\ & \rightarrow^P 0.
\end{align*}

Case $X_t = B_t^j$ $(1\leq j \leq d)$, $Y_t = W_t^i$ $(0\leq i \leq d)$: 
\begin{align*}
\sqrt{n} E\Big[\Psi^{\Hz}\Psi^B \int_0^1 \int_{\eta_n(s)}^s \tilde{\theta}_r dB_r^j dW_s^i\Big|\F_1^W\Big] 
&= \sqrt{n} \int_0^1 E\Big[ \Psi^B \int_{\eta_n(s)}^s \Psi^{\Hz} \tilde{\theta}_r dB_r^j \Big|\F_s^W\Big] dW_s^i 
\\ & = \sqrt{n} \int_0^1 \int_{\eta_n(s)}^s E[\Psi^{\Hz} f_r^j \tilde{\theta}_r |\F_r^W]dr dW_s^i
\\ & \rightarrow^P 0. 
\end{align*}

Case $X_t = W_t^j$ or $B_t^j$ $(1\leq j \leq d)$, $Y_t = B_t^i$ $(1\leq i \leq d)$: 
\begin{align*}
\sqrt{n} E\Big[\Psi^{\Hz}\Psi^B \int_0^1 \int_{\eta_n(s)}^s \tilde{\theta}_r dX_r^j dB_s^i\Big|\F_1^W\Big] 
&= \sqrt{n} E\Big[\int_0^1 \Big(\int_{\eta_n(s)}^s \Psi^{\Hz} \tilde{\theta}_r dX_r^j\Big) f_s^i ds \Big|\F_1^W\Big]. 
\end{align*}
We can prove 
\[
\sqrt{n}\int_0^1 \Big(\int_{\eta_n(s)}^s \Psi^{\Hz}\tilde{\theta}_r dX_r^j\Big) f_s^i ds \rightarrow^P 0
\] 
through the approximation of $f^i$ by simple processes 
as in the proof of \cite[Theorem 5.5]{JP}. Notice that 
$(\sqrt{n}\int_0^1 (\int_{\eta_n(s)}^s \Psi^{\Hz}\tilde{\theta}_r dX_r^j) f_s^i ds)$ is uniformly integrable. 
Consequently we have 
\[
\sqrt{n} E\Big[\int_0^1 \Big(\int_{\eta_n(s)}^s \Psi^{\Hz}\tilde{\theta}_r dX_r^j\Big) f_s^i ds \Big|\F_1^W\Big]
\rightarrow^P 0. 
\]
The proof is now complete.
\end{proof}


The results in Theorem \ref{thm:zerolimit2} imply that 
some of these sequences may converge to a non-zero random variable with higher-order scaling. 
For example, we can easily prove that, for square-integrable $F$ and $(\theta_s)_{s \in [0,1]}$, 
\[
n E\Big[F\int_0^1 \int_{\eta_n(s)}^{s}\theta_r dr ds \Big| \F_1^W\Big] \rightarrow^P \frac{1}{2} E\Big[F\int_0^1 \theta_s ds \Big| \F_1^W\Big].
\]
For a further understanding of limit theorems for conditional expectations, 
we study the asymptotic distribution for 
$n E[F\int_0^1 \int_{\eta_n(s)}^{s}\theta_r dr dW_s | \F_1^W]$, 
which can be proved in a similar way to Theorem \ref{thm:limittheorem}. 

\begin{theorem}
Let $\Hz \subset \F$ be a sub $\sigma$-field independent of $W$ and define the filtration $\ff = (\F_t)_{t \in [0,1]}$
as $\F_t := \Hz \vee \F_t^W$, $t \in [0,1]$. 
If $F = (F^1, \dots, F^q) \in L^2(\F_1; \R^q)$ and 
$(\theta_t^{ik})_{t\in[0,1]}$, $1 \leq i \leq d$, $1 \leq k \leq q$ are $\ff$-predictable processes 
with $E[\int_0^1 |\theta_s^{ik}|^2ds] < \infty$, 
then 
\begin{align*}
& \left( n E\Big[F^k \int_0^1 \int_{\eta_n(s)}^s \theta_r^{ik} dr dW_s^i \Big|\F_1^W\Big] \right)_{1 \leq i\leq d, 1 \leq k \leq q}
\\ & \Longrightarrow^{s-\F}\left( 
\frac{1}{2} E\Big[F^k \int_0^1 \theta_s^{ik} dW_s^i \Big|\F_1^W\Big]
+\frac{1}{\sqrt{12}} \int_0^1 E[F^k\theta_s^{ik}|\F_1^W] dN_s^{i} \right)_{1 \leq i\leq d, 1 \leq k \leq q}
\end{align*}
where $\{(N_t^{i})_{t\in[0,1]}: 1\leq i\leq d\}$ is a $d$-dimensional standard Brownian motion independent of $\F$.
\end{theorem}

\begin{proof}
We prove only the case $q=1$. We keep the notation for $F_m$ in the proof of Theorem \ref{thm:limittheorem}. 
Let us define the extension 
$(\hat{\Omega},\hat{\F},\hat{P}) := (\Omega,\F,P) \times \prod_{1\leq i \leq d} (\Omega^{i},\F^{i},P^{i})$ 
where, for each $1 \leq i \leq d$, $(\Omega^{i},\F^{i},P^{i})$ is the Wiener space on which $N^{i}=(N_t^{i})_{t \in [0,1]}$ 
is the canonical Brownian motion. 

By the Cauchy-Schwarz inequality, we can show that 
\[
\sup_n \Big\|E\Big[(F-F_m) n \int_0^1 \int_{\eta_n(s)}^s \theta_r^{i} dr dW_s^i\Big|\F_1^W\Big] \Big\|_1 \rightarrow 0
\]
as $m \rightarrow \infty$. 
On the other hand, by the independence of $\Hz$ and $\F_1^W$, 
\begin{align*}
n E\Big[F_m \int_0^1 \int_{\eta_n(s)}^s \theta_r^{i} dr dW_s^i\Big|\F_1^W\Big] 
= \sum_{l=1}^m \Psi_{ml}^W n \int_0^1 \int_{\eta_n(s)}^s E[\Psi_{ml}^{\Hz} \theta_r^{i} |\F_r^W] dr dW_s^i. 
\end{align*}
Let $m$ be fixed. We can compute the limit of quadratic variations for $n \int_0^1 \int_{\eta_n(s)}^s \xi_r^{li} dr dW_s^i$ 
with $\xi_r^{li} := E[\Psi_{ml}^{\Hz} \theta_r^{i} |\F_r^W]$ 
as follows
\begin{align*}
n \Big\langle \int_0^1 \int_{\eta_n(s)}^s \xi_r^{li} dr dW_s^i, W_{\cdot}^{i'} \Big\rangle_t 
&= \delta_{ii'} n \int_0^t \int_{\eta_n(s)}^s \xi_r^{li} dr ds 
\\ & \rightarrow^P \frac{1}{2} \delta_{ii'} \int_0^t \xi_s^{li} ds, 
\\ n^2 \Big\langle \int_0^1 \int_{\eta_n(s)}^s \xi_r^{li} dr dW_s^i, \int_0^1 \int_{\eta_n(s)}^s \xi_r^{l'i'} dr dW_s^{i'} \Big\rangle_t 
 & = \delta_{ii'} n^2 \int_0^t \Big(\int_{\eta_n(s)}^s \xi_r^{li}dr \int_{\eta_n(s)}^s \xi_r^{l'i'}dr \Big) ds 
\\ & \rightarrow^P \frac{1}{3} \delta_{ii'} \int_0^t \xi_s^{li}\xi_s^{l'i'}ds. 
\end{align*}
Then we obtain the following result of stable convergence from the martingale limit theorem in \cite{J97}: 
\begin{align*}
& \Big(n E\Big[F_m \int_0^1 \int_{\eta_n(s)}^s \theta_r^{i} dr dW_s^i\Big|\F_1^W\Big] \Big)_{1 \leq i \leq d}
\\ & \Longrightarrow^{s-\F} 
\Big( \sum_{l=1}^m \Psi_{ml}^W 
\Big( \frac{1}{2}\int_0^1 E[\Psi_{ml}^{\Hz} \theta_s^{i} |\F_s^W] dW_s^i 
+ \frac{1}{\sqrt{12}} \int_0^1 E[\Psi_{ml}^{\Hz} \theta_s^{i} |\F_s^W] dN_s^i \Big) \Big)_{1 \leq i \leq d}.
\end{align*}
Since $\int_0^1 E[\Psi_{ml}^{\Hz} \theta_s^{i} |\F_s^W] dW_s^i = E[\Psi_{ml}^{\Hz} \int_0^1 \theta_s^{i} dW_s^i|\F_1^W]$, the limit coincides with 
\begin{align*}
\Big(
\frac{1}{2}E\Big[F_m \int_0^1 \theta_s^{i} dW_s^i \Big|\F_1^W\Big] 
+ \frac{1}{\sqrt{12}} \int_0^1 E[F_m \theta_s^{i} |\F_1^W] dN_s^i
\Big)_{1 \leq i \leq d}.
\end{align*}
As $m \rightarrow \infty$, for each $i$, this converges (in $L^1$) to 
\[
\frac{1}{2}E\Big[F \int_0^1 \theta_s^{i} dW_s^i \Big|\F_1^W\Big] 
+ \frac{1}{\sqrt{12}} \int_0^1 E[F \theta_s^{i} |\F_1^W] dN_s^i.
\]
Therefore the proof follows from Lemma \ref{lem:sublimit}. 
\end{proof}

\section{The Euler method for nonlinear filtering}\label{sec:applications}

The goal of this section is to illustrate 
how Theorem \ref{thm:limittheorem} and \ref{thm:zerolimit2} can be applied to 
the error analysis for the Euler method for continuous-time nonlinear filtering theory. 
We shall begin with a basic formulation of the problem. 

\subsection{Settings}
Let $(\Omega, \F, \ff, P)$ be a filtered probability space 
and $\ff=(\F_t)_{t\in[0,1]}$ satisfies the usual conditions. 
Consider two stochastic processes $(X_t)_{t \in [0,1]}$ (often called the signal process) and 
$(Y_t)_{t\in[0,1]}$ (called observation process) defined as the solution of 
$e$- and $d$-dimensional stochastic differential equation 
\begin{align*}
X_t &= X_0 + \int_0^t b(X_s, Y_s)ds + \int_0^t\sigma(X_s,Y_s)dB_s + \int_0^t v(X_s,Y_s)dY_s
\\
&= X_0 + \int_0^t (b(X_s, Y_s)+v(X_s,Y_s)h(X_s,Y_s))ds + \int_0^t\sigma(X_s,Y_s)dB_s + \int_0^t v(X_s,Y_s)dW_s,
\\
Y_t &= \int_0^t h(X_s, Y_s)ds + W_t
\end{align*}
where $X_0$ is $\F_0$-measurable, 
and $B=(B_t)_{t\in[0,1]}$ and $W=(W_t)_{t\in[0,1]}$ stand for $e$- and $d$-dimensional $\ff$-Brownian motions 
independent of each other. 
The coefficients $b = (b_i)_{1 \leq i \leq e}: \R^e\times\R^d \rightarrow \R^e$, 
$\sigma = (\sigma_{ij})_{1 \leq i,j \leq e}: \R^e\times\R^d \rightarrow \R^e\otimes\R^e$, 
$v = (v_{ik})_{1 \leq i\leq e, 1 \leq k \leq d}: \R^e\times\R^d \rightarrow \R^e\otimes\R^d$, 
$h = (h_i)_{1 \leq i \leq d}: \R^e\times\R^d \rightarrow \R^d$ 
are assumed to be Lipschitz continuous. 

We are interested in the practical issues of nonlinear filtering, in particular, 
the computational problem for the conditional expectation $E[g(X_1)|\F_1^Y]$. 
One of the approaches to the problem is the change of measure approach (see e.g.\ \cite{BC09}, \cite{K11}). 
We set 
\[
\Phi_t := \exp\Big( \int_0^t h(X_{s}, Y_{s}) dY_s 
- \frac{1}{2} \int_0^t |h|^2(X_{s}, Y_{s})ds \Big) 
\]
and assume $E[\Phi_1^{-1}]=1$. 
Then, using Girsanov's theorem, we can consider 
the change of measure via the density $\frac{d\tilde{P}}{dP} = \Phi_1^{-1}$. 
That is, the new probability measure is defined as 
$\tilde{P}(A) = E[1_A \Phi_1^{-1}]$ for $A \in \F$. 
Notice that, for the random variables $Z_n$ $(n\in\N)$ and $Z$, 
$Z_n \rightarrow^P Z$ and $Z_n \rightarrow^{\tilde{P}} Z$ are equivalent. 
Under the measure $\tilde{P}$, 
$B$ and $Y$ are $e$- and $d$-dimensional $\ff$-Brownian motions, 
and hence $X_0, B, Y$ are independent. 
After the change of measure, we get the abstract Bayes formula (Kallianpur-Striebel formula)
\[
E[g(X_1)|\F_1^Y] = \frac{\tilde{E}[g(X_1)\Phi_1|\F_1^Y] }{\tilde{E}[\Phi_1|\F_1^Y]}.
\]
See e.g.\ \cite{K11} for the details. 
The advantage of this change of measure is that under $\tilde{P}$, 
the random variables inside the conditional expectation 
can be expressed as a functional of two Brownian motions $(B,Y)$. 
Hence one can construct the discrete-time approximation scheme for $\tilde{E}[g(X_1)\Phi_1|\F_1^Y]$ 
and give the precise error estimate for it using It\^o's stochastic calculus with respect to $(B,Y)$. 
In the next two subsections, 
the Euler method for $\tilde{E}[g(X_1)\Phi_1|\F_1^Y]$ is discussed. 

For simplicity, we often use the following notations:
\begin{itemize}
\item 
The summation over all indices (of possible states) 
is denoted by $\sum_i$, $\sum_{i,j}$, etc. 
\item 
For a smooth function $f(x,y)$ ($x \in \R^e$, $y\in\R^d$), 
$\nabla_x f$ and $\nabla_yf$ denote the gradient vector fields of the variables $x$ and $y$, respectively. 
Similarly, $\partial_{x_k}f$ and $\partial_{y_k}f$ are defined as $k$-th partial derivatives 
with respect to the variables $x$ and $y$, respectively. 
\end{itemize}

\subsection{Asymptotic error distribution of the Euler method I}\label{sebsec:picard}
First of all, we consider the Euler-type approximation for 
$\log(\Phi_1)$, that is, 
\[
\bar{\Phi}_t^n := \exp\Big( \int_0^t h(X_{\eta_n(s)}, Y_{\eta_n(s)}) dY_s 
- \frac{1}{2} \int_0^t |h|^2(X_{\eta_n(s)}, Y_{\eta_n(s)})ds \Big).
\]
If the finite-dimensional distribution of $(X_t)_{t\in[0,1]}$ is known exactly (under $\tilde{P}(\cdot|\F_1^Y)$), 
we can compute the approximate value $\tilde{E}[g(X_1)\bar{\Phi}_1^n|\F_1^Y]$ 
by Monte Carlo methods. 

In what follows, we analyze the asymptotic distribution of the unnormalized error 
\[
\tilde{E}[g(X_1) \Phi_1|\F_1^Y] - \tilde{E}[g(X_1)\bar{\Phi}_1^n |\F_1^Y]. 
\]
Subsequently, we also deduce the asymptotic distribution of the normalized error 
\[
E[ g(X_1)|\F_1^Y] - \frac{\tilde{E}[ g(X_1) \bar{\Phi}_1^n |\F_1^Y]}{\tilde{E}[\bar{\Phi}_1^n |\F_1^Y]}. 
\]
In order to prove several properties, the following three conditions are considered.
\begin{itemize}
\item[($A_{1}$)] : $b, \sigma, v$: Lipschitz continuous. $X_0 \in \cap_{p\geq 1} L^p$.
\item[($A_{2}$)] : $h \in C_b^1(\R^e \times \R^d; \R^d)$. 
\item[($A_{3}$)] : $g$ is a measurable function with polynomial growth. 
\end{itemize}

As an important application of Theorem \ref{thm:limittheorem}, 
we can show the following limit theorem for the unnormalized error under $\tilde{P}$ 
and the normalized error under the original measure $P$. 
If the conditional variance of the limit is not zero, 
the results imply that the order of strong convergence of these errors is at most $1/2$. 
Combining this and Proposition \ref{prop:orderestimates} described later, 
we can say that the order $1/2$ is optimal, 
except some specific cases (see e.g.~\cite{P84}, \cite{Ta86}, \cite{MT09}). 

\begin{theorem}\label{thm:app1}
Assume $(A_{1})$-$(A_{3})$. Then the following results hold. 
\begin{description}
\item[(i)] Under the probability measure $\tilde{P}$, 
\begin{align*}
\sqrt{n}\left(\tilde{E}[g(X_1) \Phi_1|\F_1^Y] - \tilde{E}[g(X_1)\bar{\Phi}_1^n |\F_1^Y]\right)
\Longrightarrow^{s-\F} 
Z \sqrt{\frac{1}{2}\sum_{1\leq i,j\leq d} \int_0^1 |u_s^{ij}(g)|^2 ds}
\end{align*}
where
\begin{align*}
u_s^{ij}(g) = \tilde{E}\Big[g(X_1)\Big(\sum_{k}\partial_{x_k}h_i(X_s, Y_s)v_{kj}(X_s,Y_s) 
+ \partial_{y_j} h_i(X_s, Y_s)\Big)\Phi_1 \Big|\F_1^Y\Big]
\end{align*}
and $Z$ is a standard normal random variable independent of $\F$. 
\item[(ii)] Under the original probability measure $P$, 
\begin{align*}
\sqrt{n}\left( E[ g(X_1)|\F_1^Y] - \frac{\tilde{E}[ g(X_1) \bar{\Phi}_1^n |\F_1^Y]}{\tilde{E}[\bar{\Phi}_1^n |\F_1^Y]} \right) 
\quad \Longrightarrow^{s-\F} 
Z 
\sqrt{\frac{1}{2}\sum_{1\leq i,j\leq d} \int_0^1 |\mu_s^{ij}(g)|^2 ds }
\end{align*}
where 
\begin{align*} 
\mu_s^{ij}(g) &= E\Big[g(X_1)\Big(\sum_{k}\partial_{x_k}h_i(X_s, Y_s)v_{kj}(X_s,Y_s) + \partial_{y_j}h_i(X_s, Y_s)\Big)\Big|\F_1^Y\Big] 
\\ & \quad - E[g(X_1)|\F_1^Y]E\Big[\sum_{k}\partial_{x_k}h_i(X_s, Y_s)v_{kj}(X_s,Y_s)+\partial_{y_j}h_i(X_s, Y_s)\Big|\F_1^Y\Big]
\end{align*}
and $Z$ is a standard normal random variable independent of $\F$. 
\end{description}
\end{theorem}

\subsection{Asymptotic error distribution of the Euler method II}\label{sebsec:general_euler}

We now consider the Euler-Maruyama scheme for the stochastic differential equation $X = (X_t)_{t\in[0,1]}$. 
The scheme $\bar{X}^n = (\bar{X}_t^n)_{t\in [0,1]}$ is defined by 
\begin{align*}
\bar{X}_t^n = X_0 + \int_0^t b(\bar{X}_{\eta_n(s)}^n, Y_{\eta_n(s)})ds 
+ \int_0^t\sigma(\bar{X}_{\eta_n(s)}^n, Y_{\eta_n(s)})dB_s + \int_0^t v(\bar{X}_{\eta_n(s)}^n, Y_{\eta_n(s)})dY_s.
\end{align*}
We redefine here the approximation of $\Phi_1$ as follows. For any continuous processes $U$ and $V$, 
\[
\bar{\Phi}_t^n(U,V) := \exp\Big( \int_0^t h(U_{\eta_n(s)}, V_{\eta_n(s)}) dY_s 
- \frac{1}{2} \int_0^t |h|^2(U_{\eta_n(s)}, V_{\eta_n(s)})ds \Big).
\]
In what follows, the pair $(\bar{X}^n, \bar{\Phi}_1^n(\bar{X}^n, Y))$ is used 
as an approximation for $(X, \Phi_1)$. 
Note that $\bar{\Phi}_t^n(X,Y) = \bar{\Phi}_t^n$ in the previous. 

The conditional expectation $\tilde{E}[g(\bar{X}_1^n)\bar{\Phi}_1^n(\bar{X}^n, Y) |\F_1^Y]$ 
is implementable by Monte Calro method for $\bar{X}_1^n$. 
Indeed, both $\bar{X}_1^n$ and $\bar{\Phi}_1^n(\bar{X}^n, Y)$ 
are functionals of $X_0$, $\{B_{(i+1)/n}-B_{i/n}\}_{0\leq i \leq n-1}$, 
$\{Y_{(i+1)/n}-Y_{i/n}\}_{0\leq i \leq n-1}$. 
This implies that
\[
\tilde{E}[g(\bar{X}_1^n)\bar{\Phi}_1^n(\bar{X}^n, Y) |\F_1^Y] 
= \tilde{E}[g(\bar{X}_1^n)\bar{\Phi}_1^n(\bar{X}^n, Y) | \{Y_{(i+1)/n}-Y_{i/n}\}_{0\leq i \leq n-1} ]. 
\]
The right hand side is the expectaion of a functional of $\bar{X}^n$ 
under a given descrete observation $\{Y_{i/n}\}_{1\leq i \leq n}$. 

In the subsection, some additional smoothness assumptions for the coefficients are required. 
\begin{itemize}
\item[($A_{1}'$)] : $b, \sigma, v$: $C^1$ with bounded partial derivatives. $X_0 \in \cap_{p\geq 1} L^p$.
\item[($A_{2}'$)] : $h \in C_b^1(\R^e \times \R^d; \R^d)$. 
\item[($A_{3}'$)] : $g \in C^1$ such that $g$ and $\nabla g$ are polynomial growth. 
\end{itemize}

We finally show a similar version of Theorem \ref{thm:app1} together with the approximation to $(X_t)_{t\in[0,1]}$. 
\begin{theorem}\label{thm:app2}
Under $(A_{1}')$-$(A_{3}')$, let $\{ (\mathcal{E}_t^{ij})_{t \in [0,1]}, 1\leq i,j\leq e+1\}$ be the solution to 
the following linear stochastic differential equation: for $1 \leq i\leq e$, 
\begin{align*}
\mathcal{E}_t^{ij} = \delta_{ij} + \sum_{k} \int_0^t \partial_{x_k} b_i(X_s,Y_s) \mathcal{E}_s^{kj} ds 
+ \sum_{k,l} \int_0^t \partial_{x_k} \sigma_{i,l}(X_s,Y_s) \mathcal{E}_s^{kj} dB_s^l 
+ \sum_{k,l} \int_0^t \partial_{x_k} v_{i,l}(X_s,Y_s) \mathcal{E}_s^{kj} dY_s^l, 
\end{align*}
and for $i=e+1$, 
\begin{align*}
\mathcal{E}_t^{ij} = \delta_{ij} 
+ \sum_{k,l} \int_0^t \partial_{x_k} h_{l}(X_s,Y_s) \mathcal{E}_s^{kj} dY_s^l
+ \sum_{k} \int_0^t \partial_{x_k} |h|^2 (X_s,Y_s) \mathcal{E}_s^{kj} ds. 
\end{align*}
Then the following results hold. 

\begin{description}
\item[{\rm (i)}] Under the probability measure $\tilde{P}$,
\begin{align*}
\sqrt{n}\left(\tilde{E}[g(X_1) \Phi_1|\F_1^Y] - \tilde{E}[g(\bar{X}_1^n)\bar{\Phi}_1^n(\bar{X}^n, Y) |\F_1^Y]\right)
\quad \Longrightarrow^{s-\F} 
Z \sqrt{\frac{1}{2}\sum_{1\leq i,j\leq d} \int_0^1 |u_s^{ij}(g)|^2 ds}, 
\end{align*}
where 
\begin{align*} 
u_s^{ij}(g) & = 
\tilde{E}\Big[\sum_{k',k''} \Big( \sum_{1\leq k \leq e}\partial_{k}g(X_1) \mathcal{E}_1^{kk'} + \mathcal{E}_1^{(e+1)k'}\Big) 
(\mathcal{E}_s^{-1})^{k'k''} f_s^{ijk''} \Phi_1 \Big|\F_1^Y\Big], 
\end{align*}
$f_s^{ijk''}$ is given by 
\begin{align*}
f_s^{ijk''} = 
\left\{ \begin{array}{ll}
    \sum_{l} \partial_{x_l}v_{k''i}(X_s,Y_s)v_{lj}(X_s,Y_s) + \partial_{y_j}v_{k''i}(X_s,Y_s), & k'' = 1, \dots, e, \\
    \sum_{l}\partial_{x_l}h_i(X_s, Y_s)v_{lj}(X_s,Y_s) + \partial_{y_j}h_i(X_s, Y_s), & k'' = e+1
  \end{array} \right.
\end{align*}
and $Z$ is a standard normal random variable independent of $\F$. 

\item[{\rm (ii)}] Under the original probability measure $P$, 
\begin{align*}
\sqrt{n}\left( E[ g(X_1)|\F_1^Y] 
- \frac{\tilde{E}[ g(\bar{X}_1^n) \bar{\Phi}_1^n(\bar{X}^n, Y) |\F_1^Y]}{\tilde{E}[\bar{\Phi}_1^n(\bar{X}^n, Y) |\F_1^Y]} \right) 
\quad \Longrightarrow^{s-\F} 
Z \sqrt{\frac{1}{2}\sum_{1\leq i,j\leq d} \int_0^1 |\mu_s^{ij}(g)|^2 ds }, 
\end{align*}
where 
\begin{align*}
\mu_s(g) &= 
E\Big[\sum_{k',k''} \Big( \sum_{1\leq k \leq e}\partial_{k}g(X_1) \mathcal{E}_1^{kk'} + \mathcal{E}_1^{(e+1)k'}\Big)
(\mathcal{E}_s^{-1})^{k'k''} f_s^{ijk''} \Big|\F_1^Y\Big]
\\ & \quad + E[g(X_1)|\F_1^Y] 
E\Big[\sum_{k',k''} \mathcal{E}_1^{(e+1)k'}(\mathcal{E}_s^{-1})^{k'k''} f_s^{ijk''} \Big|\F_1^Y\Big]
\end{align*}
and $Z$ is a standard normal random variable independent of $\F$. 
\end{description}

\begin{remark}
We see that $\mu_s = 0$ if 
\begin{align*}
\partial_y h(x,y) &= 0, 
\\ v(x,y) &= 0. 
\end{align*}
Additionally, if we assume $\partial_y b(x,y) = 0$ and $\partial_y \sigma(x,y) = 0$, 
then we get the following standard model in filtering theory 
\begin{align*}
X_t &= X_0 + \int_0^t b(X_s)ds + \int_0^t \sigma(X_s)dB_s 
\\ 
Y_t &= \int_0^t h(X_s)ds + W_t. 
\end{align*}
The asymptotic error distribution of the Euler method for this case (with order $1/n$) has not been analyzed. 
\end{remark}
\end{theorem}

\subsection{Proof of Theorem \ref{thm:app1}}
Before giving the proof of the theorem, we prove several auxiliary results.

\begin{lemma}\label{lem:add1}
If $(A_1)$-$(A_2)$ holds, then 
for any $p \in \R$, 
\[\tilde{E}[|\Phi_1|^p] + \sup_{n} \tilde{E}[|\bar{\Phi}_1^n|^p] < \infty.
\]
\end{lemma}

\begin{proof}
Note that both $\Phi_t$ and $\bar{\Phi}_t^n$ are the solutions of linear stochastic diffrential equations 
whose coeficients are bounded. Then the result follows using Gronwall's inequality.
\end{proof}

\begin{lemma}\label{lem:add2}
If $(A_1)$-$(A_2)$ holds, then 
$\sup_{0\leq t \leq 1}|X_t| + \sup_{0\leq t \leq 1}|\bar{X}_t^n| \in L^p(P) \cap L^p(\tilde{P})$ for every $p \geq 1$
\end{lemma}

\begin{proof}
Since $b, \sigma, v$ are Lipschitz continuous, we know that 
$\sup_{0\leq t \leq 1}|X_t| + \sup_{0\leq t \leq 1}|\bar{X}_t^n| \in L^p(\tilde{P})$. 
The assertion $\sup_{0\leq t \leq 1}|X_t| + \sup_{0\leq t \leq 1}|\bar{X}_t^n| \in L^p(P)$ follows from 
the change of measure and H\"older's inequality. 
\end{proof}

\begin{lemma}\label{lem:phi_conv}
Under $(A_1)$-$(A_3)$ and $h \in C_b(\R^e \times \R^d; \R^d)$, we have 
\begin{description}
\item[1.] $\bar{\Phi}_1^n \rightarrow^P \Phi_1$.
\item[2.] $\tilde{E}[g(X_1)\bar{\Phi}_1^n|\F_1^Y] \rightarrow^{P} \tilde{E}[g(X_1)\Phi_1|\F_1^Y]$.
\end{description}
\end{lemma}
\begin{proof}
We will prove two results simultaneously. 
Since $h$ and $t \mapsto (X_t, Y_t)$ are continuous, then 
$\int_0^1|h(X_s,Y_s)-h(X_{\eta_n(s)},Y_{\eta_n(s)})|^2 ds \rightarrow 0$ almost surely. 
This also implies that $\int_0^1 h(X_{\eta_n(s)},Y_{\eta_n(s)})dY_s \rightarrow^P \int_0^1 h(X_s,Y_s)dY_s$. 
Hence we get $\bar{\Phi}_1^n \rightarrow^P \Phi_1$, and also 
$\tilde{E}[g(X_1)\bar{\Phi}_1^n|\F_1^Y] \rightarrow^{P} \tilde{E}[g(X_1)\Phi_1|\F_1^Y]$ 
by Lemma \ref{lem:add1}. 
\end{proof}

\begin{lemma}\label{lem:basicineq}
If $(A_1)$, $(A_3)$ holds and $\mu \in C^1(\R^e \times \R^d)$ with bounded partial derivatives,  
then for every $p \geq 1$, 
\[
n^{p/2} \sup_{0 \leq t \leq 1} \tilde{E}\Big[\big| 
\mu(X_t,Y_t) - \mu(X_{\eta_n(t)},Y_{\eta_n(t)}) - 
\int_{\eta_n(t)}^t ( \nabla_x \mu(X_s,Y_s) dX_s + \nabla_y \mu(X_s,Y_s)dY_s) 
\big|^{p}\Big] \rightarrow 0, 
\]
Moreover, 
\begin{align*}
\tilde{E}\Big[\int_0^1 |\mu(X_t,Y_t) - \mu(X_{\eta_n(t)},Y_{\eta_n(t)})|^{p} dt \Big] \leq \frac{C(p)}{n^{p/2}}.
\end{align*}
\end{lemma}

\begin{proof}
Without loss of generality, we may assume that $\mu(x,y) = \mu(x)$ and $(X_t)$ is a one-dimensional process. 
By the mean value theorem, we have 
$
\mu(X_t) - \mu(X_{\eta_n(t)}) = \mu' (\xi_n(t)) (X_t-X_{\eta_n(t)})
$
for some $\xi_n(t)$. Therefore, we get 
\begin{align*}
\mu(X_t) - \mu(X_{\eta_n(t)}) - \int_{\eta_n(t)}^t \mu'(X_s)dX_s 
&= (\mu' (\xi_n(t))-\mu'(X_{\eta_n(t)})) (X_t-X_{\eta_n(t)}) 
\\ & \quad - \int_{\eta_n(t)}^t (\mu'(X_s)-\mu'(X_{\eta_n(s)}))dX_s. 
\end{align*}
By the assumption $(A_1)$, using Gronwall's inequality as seen in \cite[p136]{KP}, 
we can prove the following estimates
\begin{align}\label{eq:l^p_est1}
\sup_{0\leq t \leq 1}\tilde{E}[|X_t-X_{\eta_n(t)}|^{p}] \leq \frac{C(p)}{n^{p/2}}. 
\end{align}
We also obtain from the continuity of $X_t$
\begin{align}\label{eq:l^p_est2}
\tilde{E}[\sup_{0\leq t \leq 1}|\mu' (\xi_n(t))-\mu'(X_{\eta_n(t)})|^{p}] \rightarrow 0,
\quad 
\tilde{E}[\sup_{0\leq t \leq 1}|\mu' (X_t)-\mu'(X_{\eta_n(t)})|^{p}] \rightarrow 0.
\end{align}
Then the first result follows. 

The second part can be proved as follows. By the above estimates (\ref{eq:l^p_est1}) and (\ref{eq:l^p_est2}), 
\begin{align*}
& \tilde{E}\Big[\int_0^1 |\mu(X_t,Y_t) - \mu(X_{\eta_n(t)},Y_{\eta_n(t)})|^{p} dt \Big] 
\\ & \leq 
\sup_{0\leq t \leq 1}\tilde{E}\Big[ \Big| \int_{\eta_n(t)}^t (\nabla_x \mu(X_s,Y_s) dX_s + \nabla_y \mu(X_s,Y_s)dY_s)\Big|^{p} \Big] 
+ \frac{C_1(p)}{n^{p/2}}. 
\end{align*}
We can easily show that 
\[
\sup_{0\leq t \leq 1}\tilde{E}\Big[ \Big| \int_{\eta_n(t)}^t (\nabla_x \mu(X_s,Y_s) dX_s + \nabla_y \mu(X_s,Y_s)dY_s)\Big|^{p} \Big] 
\leq \frac{C_2(p)}{n^{p/2}}. 
\]
\end{proof}

Let us now introduce a basic estimate 
for the discretization of $\Phi_1$ with respect to $L^p$-norm. 

\begin{proposition}\label{prop:orderestimates}
Under $(A_{1})$-$(A_{3})$, we have 
\begin{align}
\tilde{E}[|\Phi_1-\bar{\Phi}_1^n|^p] 
\leq \frac{C(p)}{n^{p/2}}. \label{eq:sqrtestimate_density}
\end{align}
Furthermore, 
\begin{align}
\tilde{E}\Big[\big| \tilde{E}[g(X_1) \Phi_1|\F_1^Y] - \tilde{E}[g(X_1)\bar{\Phi}_1^n |\F_1^Y] \big|^p\Big]
\leq \frac{C(p,g)}{n^{p/2}}, \label{eq:sqrtestimate}
\\
E\Big[\Big|E[g(X_1)|\F_1^Y ] - 
\frac{\tilde{E}[g(X_1)\bar{\Phi}_1^n|\F_1^Y]}{\tilde{E}[\bar{\Phi}_1^n|\F_1^Y]}\Big|^p\Big]
\leq \frac{C(p,g)}{n^{p/2}}. 
\label{eq:sqrtestimate_final}
\end{align}
\end{proposition}

\begin{proof}
For every $p \geq 1$, we shall prove (\ref{eq:sqrtestimate_density}). 
We can apply the mean value theorem to $\Phi_1-\bar{\Phi}_1^n$. Then we have 
\begin{align}\label{eq:phi_meanvalue}
\Phi_1-\bar{\Phi}_1^n = \bar{\Gamma}^n \Big( 
\int_0^1(h(X_s,Y_s)-h(X_{\eta_n(s)},Y_{\eta_n(s)}))dY_s 
- \frac{1}{2}\int_0^1(|h|^2(X_s,Y_s)-|h|^2(X_{\eta_n(s)},Y_{\eta_n(s)}))ds\Big), 
\end{align}
where $\bar{\Gamma}^n = \int_0^1 \exp(\rho \log(\Phi_1)+(1-\rho)\log(\bar{\Phi}_1^n))d\rho$. 
The estimate $\sup_n\tilde{E}[|\bar{\Gamma}^n|^p]^{1/p} < \infty$ is a consequence of Lemma \ref{lem:phi_conv}.
By the Burkholder-Davis-Gundy inequality, we have 
\begin{align*}
& \tilde{E}\Big[\Big|
\int_0^1(h(X_s,Y_s)-h(X_{\eta_n(s)},Y_{\eta_n(s)}))dY_s 
- \frac{1}{2}\int_0^1(|h|^2(X_s,Y_s)-|h|^2(X_{\eta_n(s)},Y_{\eta_n(s)}))ds\Big|^p\Big] 
\\ & \leq C_p \tilde{E}\Big[ \int_0^1(|h(X_s,Y_s)-h(X_{\eta_n(s)},Y_{\eta_n(s)})|^p 
+ ||h|^2(X_s,Y_s)-|h|^2(X_{\eta_n(s)},Y_{\eta_n(s)})|^p)ds\Big].
\end{align*}
Therefore the result (\ref{eq:sqrtestimate_density}) follows from Lemma \ref{lem:basicineq} with $\mu = h$ and $\mu =|h|^2$. 
We can easily prove (\ref{eq:sqrtestimate}) using (\ref{eq:sqrtestimate_density}) and the Cauchy-Schwarz inequality. 

For the proof of (\ref{eq:sqrtestimate_final}), 
let $\rho_1(g) := \tilde{E}[g(X_1)\Phi_1|\F_1^Y]$ and $\bar{\rho}_1^n(g) := \tilde{E}[g(X_1)\bar{\Phi}_1^n|\F_1^Y]$. 
The error is expressed as 
\begin{align}\label{eq:error_decom_P}
\frac{\rho_1(g)}{\rho_1({\bf 1})} - \frac{\bar{\rho}^n_1(g)}{\bar{\rho}^n_1({\bf 1})} = 
\frac{1}{\rho_1({\bf 1})} (\rho_1(g) - \bar{\rho}^n_1(g))
- \frac{\bar{\rho}^n_1(g)}{\rho_1({\bf 1})\bar{\rho}^n_1({\bf 1})}(\rho_1({\bf1}) - \bar{\rho}^n_1({\bf 1})).
\end{align}
Then the result follows from (\ref{eq:sqrtestimate}) and the moment estimates in Lemma \ref{lem:phi_conv}. 
\end{proof}

We now prove Theorem \ref{thm:app1}. 
\begin{proof}[Proof of Theorem \ref{thm:app1} {\rm(i)}]
Using the notation in Theorem \ref{thm:limittheorem}, we shall prove 
\begin{align}\label{eq:limit_target_Sec4}
\sqrt{n}\tilde{E}[ g(X_1) (\Phi_1- \bar{\Phi}_1^n) |\F_1^Y]
\Longrightarrow^{s-\F} \frac{1}{\sqrt{2}} \sum_{i,j}\int_0^1 u_s^{ij}(g) dN_s^{ij}. 
\end{align}
The limit random variable in (\ref{eq:limit_target_Sec4}) 
has the same law as $Z (\frac{1}{2}\sum_{i,j} \int_0^1 |u_s^{ij}(g)|^2 ds)^{1/2}$ 
under the conditional probability given $\F$. 

From the expression in (\ref{eq:phi_meanvalue}), we obtain 
\begin{align*}
\sqrt{n} \tilde{E}[g(X_1) (\Phi_1- \bar{\Phi}_1^n) |\F_1^Y] 
= \tilde{E}[g(X_1) \bar{\Gamma}^n \sqrt{n} J^n |\F_1^Y] 
\end{align*}
with 
\begin{align*}
J^n := \int_0^1(h(X_s,Y_s)-h(X_{\eta_n(s)},Y_{\eta_n(s)}))dY_s 
- \frac{1}{2}\int_0^1(|h|^2(X_s,Y_s)-|h|^2(X_{\eta_n(s)},Y_{\eta_n(s)}))ds. 
\end{align*}
Applying 
$\bar{\Gamma}^n \rightarrow^P \Phi_1$ and $\tilde{E}[|\sqrt{n} J^n|^2] \leq C$,  
we can prove 
\[
\tilde{E}[g(X_1) \bar{\Gamma}^n \sqrt{n} J^n |\F_1^Y] 
= \tilde{E}[g(X_1) \Phi_1 \sqrt{n} J^n |\F_1^Y] + o_P(1).
\] 
Taking into account Theorem \ref{thm:limittheorem}, it suffices to show that 
\begin{align*}
\tilde{E}[g(X_1) \Phi_1 \sqrt{n} J^n |\F_1^Y] 
& = \sum_{i,j} \sqrt{n} 
\tilde{E}\Big[g(X_1)\Phi_1\int_0^1 \int_{\eta_n(s)}^s a_r^{ij} dY_r^j dY_s^i \Big|\F_1^Y\Big]
+o_P(1), 
\\ a_t^{ij} &= \sum_{k}\partial_{x_k}h_i(X_t, Y_t)v_{kj}(X_t,Y_t) 
+ \partial_{y_j} h_i(X_t, Y_t), 
\end{align*}
which implies the desired result (\ref{eq:limit_target_Sec4}). 
To see this, set $J^n = K^n - \frac{1}{2} L^n$ as 
\begin{align*}
K^n &:= \int_0^1(h(X_s,Y_s)-h(X_{\eta_n(s)},Y_{\eta_n(s)}))dY_s, 
\\ 
L^n &= \int_0^1(|h|^2(X_s,Y_s)-|h|^2(X_{\eta_n(s)},Y_{\eta_n(s)}))ds. 
\end{align*}
Applying the first part of Lemma \ref{lem:basicineq} to $K^n$, we get 
\begin{align*}
n \tilde{E}\Big[\Big| K^n - \sum_i \int_0^1 
\int_{\eta_n(s)}^s \Big\{ \sum_k \partial_{x_k}h_i(X_r,Y_r)dX_r^k + \sum_j \partial_{y_j}h_i(X_y,Y_r)dY_r^j\Big\}
dY_s^i \Big|^2\Big] \rightarrow 0. 
\end{align*}
Note that $dX_t^k = b_k(X_t, Y_t)dt 
+ \sum_{l} \sigma_{kl}(X_t,Y_t)dB_t^l + \sum_{j}v_{kj}(X_t,Y_t)dY_t^j$. 
Thus, by Theorem \ref{thm:zerolimit2}, we have 
\begin{align*}
\tilde{E}[g(X_1) \Phi_1 \sqrt{n} K^n |\F_1^Y] 
= \sum_{i,j} \sqrt{n} \tilde{E}\Big[g(X_1)\Phi_1\int_0^1 \int_{\eta_n(s)}^s a_r^{ij} dY_r^j dY_s^i \Big|\F_1^Y\Big] + o_P(1). 
\end{align*}
A similar calculation can be applied to $L^n$, and it follows from Theorem \ref{thm:zerolimit2} that 
\[
\tilde{E}[g(X_1) \Phi_1 \sqrt{n} L^n |\F_1^Y] \rightarrow^P 0. 
\]
This proves the assertion (i). 
\end{proof}

\begin{proof}[Proof of Theorem \ref{thm:app1} {\rm(ii)}]
Recall the expression (\ref{eq:error_decom_P}), that is, 
\begin{align*}
\frac{\rho_1(g)}{\rho_1({\bf 1})} - \frac{\bar{\rho}^n_1(g)}{\bar{\rho}^n_1({\bf 1})} = 
\frac{1}{\rho_1({\bf 1})} (\rho_1(g) - \bar{\rho}^n_1(g))
- \frac{\bar{\rho}^n_1(g)}{\rho_1({\bf 1})\bar{\rho}^n_1({\bf 1})}(\rho_1({\bf1}) - \bar{\rho}^n_1({\bf 1})).
\end{align*}
As seen in Theorem \ref{thm:limittheorem} ($q=2$) and the proof of Theorem \ref{thm:app1} (i), 
the pair $(\rho_1(g)-\bar{\rho}_1^n(g), \rho_1({\bf 1})-\bar{\rho}_1^n({\bf1}))$ has the following limit 
\[
\sqrt{n} \left(\rho_1(g)-\bar{\rho}_1^n(g), \rho_1({\bf 1})-\bar{\rho}_1^n({\bf1})\right) \Longrightarrow^{s-\F} 
\left( \frac{1}{\sqrt{2}} \sum_{i,j}\int_0^1 u_s^{ij}(g) dN_s^{ij}, 
\frac{1}{\sqrt{2}} \sum_{i,j} \int_0^1 u_s^{ij}({\bf1}) dN_s^{ij} \right).
\]
Applying the result $\rho_1(g)-\bar{\rho}_1^n(g) \rightarrow^P 0$, we can obtain 
\begin{align*}
\sqrt{n} \left(\frac{\rho_1(g)}{\rho_1({\bf 1})} - \frac{\bar{\rho}^n_1(g)}{\bar{\rho}^n_1({\bf 1})}\right)
\Longrightarrow^{s-\F} & \frac{1}{\sqrt{2}} \sum_{i,j}\int_0^1 \frac{u_s^{ij}(g)}{\rho_1({\bf1})} dN_s^{ij} 
- \frac{1}{\sqrt{2}}\frac{\rho_1(g)}{\rho_1({\bf1})} \sum_{i,j}\int_0^1 \frac{u_s^{ij}({\bf1})}{\rho_1({\bf1})} dN_s^{ij}
\\ = & \ \frac{1}{\sqrt{2}} \sum_{i,j} \int_0^1 \mu_s^{ij}(g) dN_s^{ij}.
\end{align*}
We finally remark that the stable limit still holds under the original measure $P$ 
which is absolutely continuous with respect to $\tilde{P}$.
\end{proof}

\subsection{Proof of Theorem \ref{thm:app2}}\label{subsec:proof2}
As well-known, by applying Gronwall's inequality to the function 
$s \mapsto \tilde{E}[ \sup_{0\leq t \leq s} |X_t - \bar{X}_t^n|^p]$,
the strong rate of convergence of the Euler-Maruyama scheme for the SDE with Lipschitz continuous coefficients 
is given by 
\begin{align}\label{eq:EM_estimate}
\tilde{E}\Big[ \sup_{0\leq t \leq 1} |X_t - \bar{X}_t^n|^p \Big] \leq \frac{C(p)}{n^{p/2}}. 
\end{align}
Furthermore, for any continuous and polynomial growth function $f$, we also have by (\ref{eq:EM_estimate}) 
\begin{align}\label{eq:EM_conv_p}
\sup_{0 \leq t \leq 1}|f(X_t)-f(\bar{X}_t^n)| \rightarrow 0 \quad \textrm{ in probability and in } L^p. 
\end{align}

Similarly to the proof of Lemma \ref{lem:basicineq}, 
the following $L^p$-estimate can be established 
by using the estimates (\ref{eq:EM_estimate}) and (\ref{eq:EM_conv_p}). 


\begin{lemma}\label{lem:basicineq2}
If $(A'_1)$ holds and $\mu \in C^1(\R^e \times \R^d)$ with bounded partial derivatives,  
then for every $p \geq 1$, 
\[
n^{p/2} \sup_{0 \leq t \leq 1} \tilde{E}\Big[\big| 
\mu(\bar{X}_t^n,Y_t) - \mu(\bar{X}_{\eta_n(t)}^n,Y_{\eta_n(t)}) - 
\int_{\eta_n(t)}^t ( \nabla_x \mu(X_s,Y_s) dX_s + \nabla_y \mu(X_s,Y_s)dY_s) 
\big|^{p}\Big] \rightarrow 0.
\]
\end{lemma}

We now give the $L^p$-estimates corresponding to (\ref{eq:order_1/2}) in our introduction. 
The proof is straightforward, and can be easily obtained as in Proposition \ref{prop:orderestimates}.
\begin{proposition}\label{prop:lpestimate2}
Under the assumptions $(A_{1}')$-$(A_{3}')$, we have 
\begin{align*}
\tilde{E}[|\Phi_1 - \bar{\Phi}_1^n(\bar{X}^n,Y)|^p] \leq \frac{C(p)}{n^{p/2}}.
\end{align*}
Moreover, 
\begin{align*}
\tilde{E}[|\tilde{E}[ g(X_1)\Phi_1 - g(\bar{X}_1^n)\bar{\Phi}_1^n(\bar{X}^n, Y)|\F_1^Y ]|^p] \leq \frac{C(p,g)}{n^{p/2}},  
\\
E\Big[\Big|E[g(X_1)|\F_1^Y ] 
- \frac{\tilde{E}[g(\bar{X}_1^n)\bar{\Phi}_1^n(\bar{X}^n,Y)|\F_1^Y]}{\tilde{E}[\bar{\Phi}_1^n(\bar{X}^n,Y)|\F_1^Y]}\Big|^p\Big] 
\leq \frac{C(p,g)}{n^{p/2}}. 
\end{align*}
\end{proposition}

Before proving Theorem \ref{thm:app2}, we introduce 
a useful formula for linear stochastic differential equations without proof (see e.g.\ \cite{CKL06}). 
\begin{lemma}\label{lem:linearsol}
Let $Z_t = (t, B_1^1, \dots, B_t^e, Y_t^1, \dots, Y_t^d)$, 
$G_t$ be a $q$-dimensional continuous semimartingale 
and $\{(a_{ij}^k(t)): 1\leq i,j\leq q, 1\leq k \leq 1+e+d\}$ be a bounded predictable process. 
Suppose $\{(\varphi_t^i): 1 \leq i \leq q\}$ is the solution of 
the stochastic integral equation 
\begin{align*}
\varphi_t^i = \sum_{j,k} \int_0^t a_{ij}^k(s) \varphi_s^j dZ_s^k + G_t^i, 
\end{align*}
and $\{(\psi_t^{ij}), 1 \leq i,j \leq q\}$ is the solution of 
the linear stochastic differential equation 
\begin{align*}
\psi_t^{ij} = \delta_{ij} + \sum_{l,k} \int_0^t a_{il}^k(s) \psi_s^{lj} dZ_s^k. 
\end{align*}
Then $(\varphi_t)$ can be solved as 
\begin{align*}
\varphi_t^i = \sum_{j,k} \psi_t^{ij} \int_0^t \big((\psi_s^{-1})^{jk} dG_s^k 
- \sum_{l,m} (\psi_s^{-1})^{jl}a_{lk}^m(s) d\langle Z^m, G^k \rangle_s \big).
\end{align*}
\end{lemma}

Using the above Lemma, we can prove the main result in this section. 
The essential idea of the computation is similar to \cite[Section 1]{CKL06}. 

\begin{proof}[Proof of Theorem \ref{thm:app2}]
We only prove (i) by showing that 
\begin{align*}
& \sqrt{n}\left(\tilde{E}[g(X_1) \Phi_1|\F_1^Y] - \tilde{E}[g(\bar{X}_1^n)\bar{\Phi}_1^n(\bar{X}^n, Y) |\F_1^Y]\right)
\\ &= 
\sum_{i,j} \sqrt{n} \tilde{E}\Big[ 
\sum_{k',k''} \Big( \sum_{1\leq k \leq e}\partial_{k}g(X_1) \mathcal{E}_1^{kk'} + \mathcal{E}_1^{(e+1)k'}\Big) 
\int_0^1\int_{\eta_n(s)}^s (\mathcal{E}_r^{-1})^{k'k''} f_r^{ijk''} dY_r^j dY_s^i \Big|\F_1^Y\Big] + o_P(1).
\end{align*}
The result (ii) can be obtained from the same argument in Theorem \ref{thm:app1} (ii). 

Let us first define an $(e+1)$-dimensional process $\tilde{X}_t$ by 
\begin{align*}
\tilde{X}_t^i = 
\left\{ \begin{array}{ll}
    X_t^i & i = 1, \dots, e, \\
    \int_0^t h(X_s,Y_s)dY_s - \frac{1}{2}\int_0^t |h|^2(X_s,Y_s)ds,  & i = e+1
  \end{array} \right.
\end{align*}
and its approximation $\hat{X}_t^n$ by 
\begin{align*}
\hat{X}_t^{n,i} = 
\left\{ \begin{array}{ll}
    \bar{X}_t^{n,i} & i = 1, \dots, e, \\
    \int_0^t h(X_{\eta_n(s)},Y_{\eta_n(s)})dY_s - \frac{1}{2}\int_0^t |h|^2(X_{\eta_n(s)},Y_{\eta_n(s)})ds, & i = e+1.
  \end{array} \right.
\end{align*}
Set $Z_t = (t, B_1^1, \dots, B_t^e, Y_t^1, \dots, Y_t^d)$. 
For each $1 \leq k \leq e+1$ and $1 \leq m \leq e+d+1$, 
define $A_{k}^m(x,y): \R^{e}\times\R^d \rightarrow \R$ by 
\begin{align*}
A_{k}^m(x,y) = 
\left\{ \begin{array}{ll}
    b_k(x,y), & 1\leq k \leq e, \ m=0, \\
    -\frac{1}{2}|h|^2(x,y), & k=e+1, \ m=0, \\
    \sigma_{km}(x,y), & 1\leq k \leq e, \ 1\leq m \leq e, \\
    v_{k(m-e)}(x,y), & 1\leq k \leq e, \ e+1\leq m \leq d+e+1, \\
    h_{(m-e)}(x,y), & k=e+1, \ e+1\leq m \leq d+e+1, \\
    0, & \textrm{otherwise}.
  \end{array} \right.
\end{align*}
Then $\tilde{X}_t$ solves 
\[
d\tilde{X}_t^k = \sum_m A_k^m(X_t, Y_t)dZ_t^m. 
\]
The function $\tilde{g}: \R^{e+1} \rightarrow \R$ is defined as 
\[
\tilde{g}(x_1, \dots, x_e, x_{e+1}) = g(x_1,\dots,x_e)\exp(x_{e+1}). 
\]
Substituting the above into the target $\tilde{E}[g(X_1) \Phi_1|\F_1^Y] - \tilde{E}[g(\bar{X}_1^n)\bar{\Phi}_1^n(\bar{X}^n, Y) |\F_1^Y]$, 
we have 
\begin{align*}
C_n 
&:= \sqrt{n} \left(\tilde{E}[g(X_1) \Phi_1|\F_1^Y] - \tilde{E}[g(\bar{X}_1^n)\bar{\Phi}_1^n(\bar{X}^n, Y) |\F_1^Y] \right)
\\ &= \sqrt{n} \tilde{E}[ \tilde{g}(\tilde{X}_1) - \tilde{g}(\hat{X}_1^n) |\F_1^Y]. 
\end{align*}
By the mean value theorem for $\tilde{g}$, we can choose $\xi_n$ such that 
\begin{align*}
\tilde{g}(\tilde{X}_1) - \tilde{g}(\hat{X}_1^n) =  
\sum_k (\partial_k \tilde{g})(\xi_n) (\tilde{X}_1^k - \hat{X}_1^{n,k}).
\end{align*}
To see the asymptotic error $\sqrt{n}(\tilde{g}(\tilde{X}_1) - \tilde{g}(\hat{X}_1^n))$, 
we consider the $L^1$-estimate below:
\begin{align*}
& \|((\partial_k \tilde{g})(\tilde{X}_1)-(\partial_k \tilde{g})(\xi_n))\sqrt{n}(\tilde{X}_1^k - \hat{X}_1^{n,k})\|_1
\\ & \leq \| (\partial_k \tilde{g})(\tilde{X}_1)-(\partial_k \tilde{g})(\xi_n)) \|_2
\| \sqrt{n} (\tilde{X}_1^k - \hat{X}_1^{n,k}) \|_2
\end{align*}
Since $|\xi_n|\leq \max\{|\tilde{X}_1|, |\hat{X}_1^n|\}$ and $\partial_k \tilde{g}$ is polynomial growth, 
$((\partial_k \tilde{g})(\tilde{X}_1)-(\partial_k \tilde{g})(\xi_n))^2$ is uniformly integrable. 
Therefore, using $\xi_n \rightarrow^P \tilde{X}_1$, we have
\[
\| (\partial_k \tilde{g})(\tilde{X}_1)-(\partial_k \tilde{g})(\xi_n)) \|_2 \rightarrow 0.
\]
By (\ref{eq:EM_estimate}), we also have $\| \sqrt{n} (\tilde{X}_1^k - \hat{X}_1^{n,k}) \|_2 \leq C$. 
Consequently, we obtain
\[
\|((\partial_k \tilde{g})(\tilde{X}_1)-(\partial_k \tilde{g})(\xi_n))\sqrt{n}(\tilde{X}_1^k - \hat{X}_1^{n,k})\|_1 \rightarrow 0. 
\]
From this, we conclude that
\begin{align*}
C_n = \sum_k \tilde{E}[ (\partial_k \tilde{g})(\tilde{X}_1) \sqrt{n} (\tilde{X}_1^k - \hat{X}_1^{n,k}) |\F_1^Y] + o_P(1). 
\end{align*}

Set $\varphi_t = \tilde{X}_t - \hat{X}_t^n$. 
Let us consider the SDE 
\begin{align*}
\varphi_t^k = \sum_{k',m} \int_0^t a_{kk'}^m(s) \varphi_s^{k'} dZ_s^{m} + G_t^k, \quad 1 \leq k \leq e+1, 
\end{align*}
where $a_{kk'}^m(s) = \partial_{x_{k'}}A_k^m(X_t,Y_t)$ and 
\begin{align}\label{eq:remainder_terms_all}
G_t^k 
&= R_t^{k,n} + R_t^{D,k,n} + R_t^{B,k,n} + R_t^{Y,k,n}.
\end{align}
The error terms $R_t^{k,n}, R_t^{D,k,n}, R_t^{B,k,n}, R_t^{Y,k,n}$ are given as follows: 
\begin{align*}
&R_t^{k,n} = \sum_m \int_0^t 
( A_k^m(X_s,Y_s)-A_k^m(\bar{X}_s^n,Y_s) - \sum_{l} a_{k,l}^m(s) (X_s^l - \bar{X}_s^{n,l})
)dZ_s^m, 
\\
&R_t^{D,k,n} = \int_0^t (A_k^0(\bar{X}_s^n,Y_s)-A_k^0(\bar{X}_{\eta_n(s)}^n,Y_{\eta_n(s)})) ds, 
\\
&R_t^{B,k,n} = \sum_{1\leq m \leq e} \int_0^t (A_k^m(\bar{X}_s^n,Y_s)-A_k^m(\bar{X}_{\eta_n(s)}^n,Y_{\eta_n(s)})) dZ_s^m, 
\\
&R_t^{Y,k,n} = \sum_{e+1\leq m \leq e+d+1} \int_0^t (A_k^m(\bar{X}_s^n,Y_s)-A_k^m(\bar{X}_{\eta_n(s)}^n,Y_{\eta_n(s)})) dZ_s^m. 
\end{align*}
Applying Lemma \ref{lem:linearsol} to $\varphi_t = \tilde{X}_t - \hat{X}_t^n$, we have 
\begin{align*}
\varphi_1^k = 
\sum_{k',k''} \mathcal{E}_1^{kk'} \int_0^1 \big((\mathcal{E}_s^{-1})^{k'k''} dG_s^{k''} 
- \sum_{l,m} (\mathcal{E}_s^{-1})^{k'l}a_{lk''}^m(s) d\langle Z^m, G^{k''} \rangle_s \big).
\end{align*}
Hence it suffices to show the following three properties to prove our assertion in the theorem. 
\begin{description}
\item[(a)] For each $k,k',k''$, 
\begin{align*}
C_n^Y & := \sqrt{n} \tilde{E}\Big[ (\partial_k \tilde{g})(\tilde{X}_1) 
\mathcal{E}_1^{kk'} \int_0^1 (\mathcal{E}_s^{-1})^{k'k''} dR_s^{Y,k'',n}\Big|\F_1^Y\Big] 
\\ & = \sum_{i,j} 
\sqrt{n} \tilde{E}\Big[ 
(\partial_k \tilde{g})(\tilde{X}_1) \mathcal{E}_1^{kk'}
\int_0^1\int_{\eta_n(s)}^s (\mathcal{E}_r^{-1})^{k'k''} f_r^{ijk''} dY_r^j dY_s^i \Big|\F_1^Y\Big] + o_P(1).
\end{align*}

\item[(b)] For each $k,k',k''$, 
\begin{align*}
\sqrt{n} \tilde{E}\Big[ (\partial_k \tilde{g})(\tilde{X}_1) 
\mathcal{E}_1^{kk'} \int_0^1 (\mathcal{E}_s^{-1})^{k'k''} (dG_s^{k''}-dR_s^{Y,k'',n})\Big|\F_1^Y\Big] \rightarrow^P 0.
\end{align*}

\item[(c)] For each $k,k',k'',l,m$, 
\begin{align*}
\sqrt{n} \tilde{E}\Big[ (\partial_k \tilde{g})(\tilde{X}_1) 
\mathcal{E}_1^{kk'} \int_0^1 (\mathcal{E}_s^{-1})^{k'l}a_{lk''}^m(s) d\langle Z^m, G^{k''} \rangle_s \Big|\F_1^Y\Big] 
\rightarrow^P 0.
\end{align*}
\end{description}

{\bf Proof of Property (a)}: By using Lemma \ref{lem:basicineq2}, $C_n^Y$ is written by 
\begin{align*}
C_n^Y =  \sqrt{n} \tilde{E}\Big[ &(\partial_k \tilde{g})(\tilde{X}_1) 
\mathcal{E}_1^{kk'} \sum_{e+1 \leq m \leq e+d+1} \int_0^1 (\mathcal{E}_s^{-1})^{k'k''} \int_{\eta_n(s)}^s 
\\ & 
\big( \partial_x A_{k''}^m(X_r,Y_r)dX_r +  \partial_y A_{k''}^m(X_r,Y_r)dY_r \big)  dZ_s^m \Big|\F_1^Y\Big] 
+ o_P(1). 
\end{align*}
Since $s \mapsto (\mathcal{E}_s^{-1})^{k'k''}$ is continuous, we can now apply the argument (\ref{eq:diff_integ}) 
in Remark \ref{rem:important} to the term 
$(\mathcal{E}_s^{-1})^{k'k''}\int_{\eta_n(s)}^s$. Then we get 
\begin{align*}
C_n^Y = \sqrt{n} \tilde{E}\Big[ & (\partial_k \tilde{g})(\tilde{X}_1) 
\mathcal{E}_1^{kk'} \sum_{e+1 \leq m \leq e+d+1} \int_0^1 \int_{\eta_n(s)}^s (\mathcal{E}_r^{-1})^{k'k''} 
\\ & 
\big( \partial_x A_{k''}^m(X_r,Y_r)dX_r +  \partial_y A_{k''}^m(X_r,Y_r)dY_r \big)  dZ_s^m \Big|\F_1^Y\Big] 
+ o_P(1). 
\end{align*}
Applying Theorem \ref{thm:zerolimit2} to $A_{k''}^m(X_r,Y_r)dX_r$, we deduce 
\begin{align*}
C_n^Y = \sqrt{n} \tilde{E}\Big[ & (\partial_k \tilde{g})(\tilde{X}_1) 
\mathcal{E}_1^{kk'} \sum_{e+1 \leq m \leq e+d+1} \int_0^1 \int_{\eta_n(s)}^s (\mathcal{E}_r^{-1})^{k'k''} 
\\ & 
\sum_j \big( \sum_l \partial_{x_l} A_{k''}^m(X_r,Y_r)v_{lj}(X_r,Y_r) dY_r^j +  \partial_{y_j} A_{k''}^m(X_r,Y_r)dY_r^j \big)  dZ_s^m \Big|\F_1^Y\Big] 
+ o_P(1). 
\end{align*}
Hence we prove the validity of Property (a) by the definition of $A_{k''}^m$. 

{\bf Proof of Property (b)}: We begin with the estimate for $R_t^{k'',n}$. 
By the mean value theorem, there exists $(\xi_k^{l,n}(s))$ such that 
\[
A_k^0(X_s,Y_s) - A_k^0(\bar{X}_s^n,Y_s) = \sum_{l}\partial_{x_l}A_k(\xi_k^{l,n}(s),Y_s)(X_s^l-\bar{X}_s^{n,l})
\]
for every $k$. Then by H\"older's inequality, 
\begin{align*}
& \tilde{E}\Big[ \Big| \sqrt{n} (\partial_k \tilde{g})(\tilde{X}_1) 
\mathcal{E}_1^{kk'} \int_0^1 (\mathcal{E}_s^{-1})^{k'k''} dR_s^{k'',n}\Big|\Big]
\\ & \leq C_1 \tilde{E}\Big[ \Big| \sqrt{n}\int_0^1 (\mathcal{E}_s^{-1})^{k'k''} dR_s^{k'',n}\Big|^2\Big]^{1/2}
\\ & = C_1 \tilde{E}\Big[ \int_0^1 |\sqrt{n}(\mathcal{E}_s^{-1})^{k'k''} 
\sum_l (\partial_{x_l}A_{k''}^0(\xi_{k''}^{l,n}(t),Y_t) - \partial_{x_l}A_{k''}^0(X_t,Y_t)) 
(X_s^l-\bar{X}_s^{n,l}) |^2 ds \Big]^{1/2}
\\ & \leq C_2 \tilde{E}\Big[\sup_{0\leq t \leq 1}|\sqrt{n}(X_t-\bar{X}_t^n)|^6\Big]^{\frac{1}{6}} 
\sum_{l}\tilde{E}\Big[\sup_{0\leq t \leq 1}
|\partial_{x_l}A_{k''}^0(\xi_{k''}^{l,n}(t),Y_t) - \partial_{x_l}A_{k''}^0(X_t,Y_t)|^6\Big]^{\frac{1}{6}}.
\end{align*}
Thus we have, using (\ref{eq:EM_estimate}) and (\ref{eq:EM_conv_p}), 
\[
\sqrt{n} \tilde{E}\Big[ (\partial_k \tilde{g})(\tilde{X}_1) 
\mathcal{E}_1^{kk'} \int_0^1 (\mathcal{E}_s^{-1})^{k'k''} dR_s^{k'',n}\Big|\F_1^Y\Big] \rightarrow^P 0.
\]
We next prove the result for $R_t^{D,k'',n}$. 
By a similar calculus to the case of Property (a), 
\begin{align*}
& \sqrt{n} \tilde{E}\Big[ (\partial_k \tilde{g})(\tilde{X}_1) 
\mathcal{E}_1^{kk'} \int_0^1 (\mathcal{E}_s^{-1})^{k'k''} dR_s^{D,k'',n}\Big|\F_1^Y\Big] 
\\ &= \sqrt{n} \tilde{E}\Big[ (\partial_k \tilde{g})(\tilde{X}_1) 
\mathcal{E}_1^{kk'} \int_0^1 \int_{\eta_n(s)}^s (\mathcal{E}_r^{-1})^{k'k''} 
\big( \partial_x A_{k''}^m(X_r,Y_r)dX_r +  \partial_y A_{k''}^m(X_r,Y_r)dY_r \big) ds\Big|\F_1^Y\Big] + o_P(1). 
\end{align*}
According to Theorem \ref{thm:zerolimit2}, this yields 
\[
\sqrt{n} \tilde{E}\Big[ (\partial_k \tilde{g})(\tilde{X}_1) 
\mathcal{E}_1^{kk'} \int_0^1 (\mathcal{E}_s^{-1})^{k'k''} dR_s^{D,k'',n}\Big|\F_1^Y\Big] \rightarrow^P 0.
\]
For $dR_t^{B,k'',n}$, by the same discussion with the above, we can also show that 
\[
\sqrt{n} \tilde{E}\Big[ (\partial_k \tilde{g})(\tilde{X}_1) 
\mathcal{E}_1^{kk'} \int_0^1 (\mathcal{E}_s^{-1})^{k'k''} dR_s^{B,k'',n}\Big|\F_1^Y\Big] \rightarrow^P 0.
\]
Consequently we obtain Property (b). 

{\bf Proof of Property (c)}: 
Notice that $d\langle Z^m, G^{k''} \rangle_s$ can be expressed by 
the form $\sum_j \int_{\eta_n(s)}^s\theta_r^j dZ_r^jds$ plus some remainder term. 
Taking into account the continuity of $s \mapsto (\mathcal{E}_s^{-1})^{k'l}a_{lk''}^m(s)$, 
we can prove Property (c) as a consequence of Theorem \ref{thm:zerolimit2} 
in addition to the estimates for the remainder term as seen in the validity of Property (a), (b). 
\end{proof}

\end{document}